\documentclass[12pt,reqno]{amsart}
\usepackage{amsmath, amssymb, amsthm, amsfonts} 
\usepackage{url}
\usepackage[breaklinks]{hyperref}
\usepackage{tikz}
\usetikzlibrary{decorations.markings}
\usepackage{array}
\renewcommand{\Re}{\operatorname{Re}}

\renewcommand{\(}{\left\(}
\renewcommand{\)}{\right\)}
\renewcommand{\[}{\left\[}
\renewcommand{\]}{\right\]}

\numberwithin{equation}{section}
 \theoremstyle{plain}
\newtheorem{theorem}{Theorem}[section]
\newtheorem{lemma}[theorem]{Lemma}
\newtheorem{remark}[]{Remark}

\newtheorem{corollary}[theorem]{Corollary}

   \makeatletter
\def\proof{\@ifnextchar[{\@oproof}{\@nproof}}
\def\@oproof[#1][#2]{\trivlist\item[\hskip\labelsep\textit{#2 Proof of\
#1.}~]\ignorespaces}
\def\@nproof{\trivlist\item[\hskip\labelsep\textit{Proof.}~]\ignorespaces}

\makeatother

\begin{document}
\title[A Number Field Analogue of Ramanujan's identity for $\zeta(2m+1)$]{A Number Field Analogue of Ramanujan's identity for $\zeta(2m+1)$} 

\author{Diksha Rani Bansal}
\address{Diksha Rani\\ Department of Mathematics \\
Indian Institute of Technology Indore \\
Indore, Simrol, Madhya Pradesh 453552, India.} 
\email{dikshaba1233@gmail.com,  mscphd2207141001@iiti.ac.in}

\author{Bibekananda Maji}
\address{Bibekananda Maji\\ Discipline of Mathematics \\
Indian Institute of Technology Indore \\
Indore, Simrol, Madhya Pradesh 453552, India.} 
\email{bibek10iitb@gmail.com, bmaji@iiti.ac.in}

\thanks{2020 \textit{Mathematics Subject Classification.} Primary 11M06,  11R42; Secondary 11R29.\\
\textit{Keywords and phrases.} Eisenstein series,  Riemann zeta function,  Odd zeta values,  Dedekind zeta function,  Ramanujan's formula,  Class number.}

\maketitle
\begin{center}
\emph{Dedicated to Professor Bruce Berndt on the occasion of his 85th birthday}
\end{center}

\begin{abstract}
Ramanujan's famous formula for $\zeta(2m+1)$ has captivated the attention of numerous mathematicians over the years.  Grosswald, in 1972, found a simple extension of Ramanujan's formula which in turn gives transformation formula for Eisenstein series over the full modular group. Recently, Banerjee, Gupta and Kumar found a number field analogue of Ramanujan's formula. 
In this paper, we present a new number field analogue of the Ramanujan-Grosswald formula for $\zeta(2m+1)$ by obtaining a formula for Dedekind zeta function at odd arguments.
We also obtain a number field analogue of an identity of Chandrasekharan and Narasimhan,  which played a crucial role in proving our main identity.  As an application,  we generalize transformation formula for Eisenstein series $G_{2k}(z)$ and Dedekind eta function $\eta(z)$.  A new formula for the class number of a totally real number field is also obtained,  which provides a connection with the Kronceker's limit formula for the Dedekind zeta function.

%   Through this framework, we derive transformation formulas for the generalized Eisenstein series, thereby extending the reach of the original Ramanujan-Grosswald formula. Furthermore, our newly derived identity encompasses both the Ramanujan and Grosswald identities as special cases.
%
%
%
% Recently, Banerjee, Gupta and Kumar \cite{BGK23} have also studied transformation formulas of zeta function at odd integers for arbitrary number field. But our work is different from their work in the sense that they have used a different generalization of Lambert series than ours.
\end{abstract}

\section{Introduction}
The theory of the Riemann zeta function $\zeta(s)$ is a central object of study in number theory that holds a substantial place in the mathematical landscape. The nature of special values of $\zeta(s)$ has a rich history.  Euler gave an exact evaluation for $\zeta(2k)$ in 1734, which establishes a relation between Bernoulli numbers and even zeta values.  More precisely, for every positive integer $k$, 
\begin{align}\label{Euler's formula for even zeta values}
\zeta(2k) = (-1)^{k+1}\frac{(2\pi)^{2k}B_{2k}}{2(2k)!},
\end{align}
where $B_{2k}$ denotes $2k$th Bernoulli number.  
The above formula instantly implies that even zeta values are transcendental. 
However,  we have very little information about algebraic nature of positive odd zeta values.   In 1979, Apery \cite{Apery79, Apery2} achieved a breakthrough by proving the irrationality of $\zeta(3)$. In 2001, Rivoal \cite{rivoal}, Ball and Rivoal \cite{ballrivoal} proved that there exist infinitely many odd zeta values which are irrational. A result due to Zudilin \cite{zudilin}, states that at least one of $\zeta(5), \zeta(7), \zeta(9)$ and $\zeta(11)$ are irrational, which is the most notable achievement in this area as of now.
Prior to all these results,  Ramanujan in his Notebook  \cite[p.~173,  Ch.~14,  Entry 21(i)]{rn2}  as well as in Lost Notebook \cite[p.~319,  Entry (28)]{lnb} gave an intriguing formula involving odd zeta values, that has drawn the attention of many mathematicians. 
For any $\alpha, \beta > 0$ with $\alpha\beta = \pi^2, k \in \mathbb{Z} \setminus \{0\}$, we have 
\begin{align}\label{Ramanujan's Formula}
H_{2k+1}(\alpha) + (-1)^{k+1} H_{2k+1}(\beta) = \sum^{k+1}_{j=0}(-1)^{j-1}\frac{B_{2j}}{(2j)!}\frac{B_{2k + 2 - 2j}}{(2k + 2 - 2j)!}\alpha^{k+1-j}\beta^j,
%&(4\alpha)^{-k} \left( \frac{1}{2}\zeta(2k+1) + \sum_{n=1}^\infty \sigma_{-(2k+1)}(n)e^{-2\alpha n} \right) - (-4\beta)^{-k} \left( \frac{1}{2}\zeta(2k+1) + \sum_{n=1}^\infty \sigma_{-(2k+1)}(n)e^{-2\beta n} \right) \nonumber \\
%& = \sum^{k+1}_{j=0}(-1)^{j-1}\frac{B_{2j}}{(2j)!}\frac{B_{2k + 2 - 2j}}{(2k + 2 - 2j)!}\alpha^{k+1-j}\beta^j,
\end{align}
where $$H_{2k+1}(x) = (4x)^{-k} \left( \frac{1}{2}\zeta(2k+1) + \sum_{n=1}^\infty \sigma_{-(2k+1)}(n)e^{-2xn} \right),$$ 
and the generalized divisor function $\sigma_{z}(n) = \sum_{d|n}d^{z}, z \in \mathbb{C}.$

Over the years, numerous mathematicians have generalized Ramanujan's formula in a variety of ways.  Ramanujan \cite[Ch.~14,  Entry 8(iii)]{rn2} himself gave a huge generalization of \eqref{Ramanujan's Formula}.  An analogue of Ramanujan's formula \eqref{Ramanujan's Formula} for $L$-functions associated to modular forms was discussed by Razar \cite[Theorem 2]{Razar} and Weil \cite{Weil} independently.   An extension for Dirichlet $L$-function,  Lerch zeta function,  and more generally for any Dirichlet series with periodic coefficients was given by  Bradley \cite{Bradley} in 2002.  Quite surprisingly,  in 1977,  Berndt \cite{berndt3} showed that Euler's formula \eqref{Euler's formula for even zeta values} and Ramanujan's formula \eqref{Ramanujan's Formula} are branches of a bigger tree,  that is,  they can be obtained from a single transformation formula for a generalized Eisenstien series.  Recently,  Dixit and the second author \cite{DM20} found an interesting one variable generalization of  \eqref{Ramanujan's Formula},  and later Dixit et.  al.  \cite{DGKM20} and Chavan \cite{Cha23} established different generalizations of  \eqref{Ramanujan's Formula} for the Hurwitz zeta function. 
To know more detailed information about the Ramanujan's formula \eqref{Ramanujan's Formula},  we refer to \cite[p.~276]{BCB-II} and a survey article by Berndt and Straub \cite{berndtstraubzeta},  and an expository paper by Dixit \cite{Dixit24} where one can find recent developments.  
Readers are also encouraged to see \cite{berndt1, berndt2,  Chavan,   CJM23,  GMR2011,  GM22,  GJKM24}.  
%{Add one line for L-functions.}

%Recently, Dixit and the second author \cite{DM20} established an elegant generalization of \eqref{Ramanujan's Formula} while extending an identity of Kanemitsu et.~al. \cite{KTY01}.  Chourasiya, Jamal and the second author \cite{CJM23} also found a new Ramanujan-type identity for Dirichlet-L function. Berndt \cite{berndt1, berndt2} and Bradley \cite{Bradley} studied character analogues of Ramanujan's formula. 

Ramanujan's formula \eqref{Ramanujan's Formula} exhibits a profound correlation with Eisenstein series.
%It has a deep connection with the fundamental transformation properties of Eisenstein series over $SL_2(\mathbb{Z})$ and their Eichler integrals. 
Let $\mathbb{H}$ be the upper half plane.  For $z \in \mathbb{H}$ and an integer $k \geq 2$,
we define the holomorphic Eisenstein series $G_{2k}(z)$ of weight $2k$ for the full modular group $\textrm{SL}_2(\mathbb{Z})$,  
\begin{align*}
G_{2k}(z) = \sum_{(m, n) \in \mathbb{Z}^2 \setminus \{(0, 0)\}}\frac{1}{(m+n z)^{2k}}.
\end{align*}
%This series converges absolutely to a holomorphic function of $\tau$  in the upper half-plane. 
%It is invariant over $SL_2(\mathbb{Z})$.  
For $a, b, c, d \in \mathbb{Z}$ with $ad-bc = 1$,  it satisfies the following modular transformation:
\begin{align}\label{trans_G_2k}
G_{2k}\left(\frac{a z + b}{c z + d}\right) = (c z + d)^{2k}G_{2k}( z ).
\end{align} 
%Define $q = e^{2\pi i z }$. 
Then the Fourier series expansion of $G_{2k}(z)$ is given by 
\begin{align*}
G_{2k}(z) = 2\zeta(2k)\left( 1 + \frac{2}{\zeta(1-2k)} \sum_{n=1}^{\infty} \sigma_{2k-1}(n)e^{2 \pi i n z}\right). 
\end{align*}
%where $c_{2k} = \frac{2}{\zeta(1-2k)}$. 
 In particular,  $G_{2k}(z)$ satisfies the following two main transformation formulae,  
\begin{align}
G_{2k}(z+1) &= G_{2k}(z), \label{Transformation1}\\
G_{2k}\left(-\frac{1}{z}\right) &= z^{2k}G_{2k}(z), \label{Transformation2}
\end{align}
which yield \eqref{trans_G_2k}.  
We now introduce a simple extension of \eqref{Ramanujan's Formula} given by Grosswald \cite{Grosswald},  in 1972, which shows how the above transformation formula \eqref{Transformation2} of $G_{2k}(z)$ is connected to Ramanujan's identity.  It states that for $z \in \mathbb{H}$, $k \in \mathbb{Z} \setminus \{0\}$,  
\begin{align}\label{Grosswald identity}
F_{2k+1}(z) - z^{2k}F_{2k+1}\left(-\frac{1}{z}\right) = \frac{1}{2}\zeta(2k+1)(z^{2k} - 1) + \frac{(2\pi i)^{2k+1}}{2z}R_{2k+1}(z),
\end{align}
where 
\begin{align}\label{infinite series_Grosswald}
F_k(z) = \sum_{n=1}^{\infty}\sigma_{-k}(n)e^{2\pi inz}
\end{align}
and $R_{2k+1}(z)$ is the Ramanujan polynomial,  introduced by Gun, Murty and Rath \cite{GMR2011},  defined as
\begin{align}\label{Ramanujan polynomial}
R_{2k+1}(z) = \sum_{j=0}^{k+1}z^{2k+2-2j}\frac{B_{2j}}{(2j)!}\frac{B_{2k + 2 - 2j}}{(2k + 2 - 2j)!}.
\end{align}
More about the above polynomial \eqref{Ramanujan polynomial} can be seen in the paper by Murty, Smyth and Wang \cite{MSW2011}.
Setting $z= i\beta/\pi$, $\alpha\beta = \pi^2,$ with $\alpha, \beta > 0$, Grosswald's identity immediately gives Ramanujan's formula \eqref{Ramanujan's Formula}. If $k < -1$ in \eqref{Grosswald identity}, then
\begin{align}\label{Grosswald's id without finite sum}
F_{2k+1}(z) - z^{2k}F_{2k+1}\left(-\frac{1}{z}\right) = \frac{1}{2}\zeta(2k+1)(z^{2k} - 1).
\end{align}
%Letting $q = e^{2\pi i z}$ in \eqref{Grosswald's id without finite sum}, 
One can easily check that the above formula is nothing but the transformation formula \eqref{Transformation2} of the Eisenstien series $G_{2k}(z)$.  We further note that the above formula is equivalent to an identity of Ramanujan \cite[p.~261,  Entry 13]{BCB-II}.

%One can see that,  for odd negative integer $k \leq -3$,   $F_{k}(z)$ is connected with the Eisenstein series $G_{1-k}(z)$. Ramanujan's formula gives both the transformation identities of $G_{2k}$. 
In the literature,  there are several generalizations of $\zeta(s)$,  however,  the Dedekind zeta function is considered as one of the most noteworthy generalizations of $\zeta(s)$ to the number fields.  In this paper,  our main aim is to establish a new generalization of Ramanujan's formula \eqref{Ramanujan's Formula} for Dedekind zeta function.  Before delving deeper, we first define the Dedekind zeta function.
%We can obtain the transformation formulas for Eisenstein series by Ramamnujan's formula as follows: \\
%In \eqref{Ramanujan's Formula}, when $k < -1,$   the finite sum (including Bernoulli numbers) is zero, so we have,
%\begin{align}\label{Ramanujan's Formula without finite series}
%(4\alpha)^{-k} \left( \frac{1}{2}\zeta(1+2k) + \sum_{n=1}^\infty \frac{n^{-2k-1}}{e^{2\alpha n} - 1} \right) = (-4\beta)^{-k} \left( \frac{1}{2}\zeta(1+2k) + \sum_{n=1}^\infty \frac{n^{-2k-1}}{e^{2\beta n} - 1} \right).
%\end{align}
%Substituting $\alpha = \pi y, \beta = \frac{\pi}{y}, k = -k$, where $y>0$ in \eqref{Ramanujan's Formula without finite series} and extending it for $\Re(y) > 0 \, \, (\text{hence} \, \, z = iy \in \mathbb{H})$ gives,
%\begin{align}
%\left( 1 + \frac{2}{\zeta(1-2k)}\sum_{n=1}^{\infty}\sigma_{2k-1}(n)e^{2\pi inz} \right) = z^{-2k}\left( 1 + \frac{2}{\zeta(1-2k)}\sum_{n=1}^{\infty}\sigma_{2k-1}(n)e^{-\frac{2\pi in}{z}} \right).
%\end{align}
%We can obtain the following transformation formula for Eisenstein series from \eqref{Ramanujan's Formula without finite series}. 
%\begin{align}
%z^{2k}E_{2k}(z) = E_{2k}\left(-\frac{1}{z}\right), \quad z \in \mathbb{H}.
%\end{align}

\subsubsection{Dedekind zeta function}
%\begin{definition}{{\bf (Dedekind zeta function)}}
Let $\mathbb{F}$ be a number field with the degree $[\mathbb{F} : \mathbb{Q}] = d = r_1 + 2r_2$, where 
$r_1$ and $r_2$ denote the number of real and complex embeddings (upto conjugates) of $\mathbb{F}$. Let $D$ be the absolute value of the discriminant $D_\mathbb{F}$ of $\mathbb{F}$.  Let $\mathcal{O}_{\mathbb{F}}$ be the ring of integers and $\mathfrak{N}$ be the norm map of $\mathbb{F}$ over $\mathbb{Q}$.  Let $\mathtt{a}_{\mathbb{F}}(n)$ be the number of ideals in $\mathcal{O}_{\mathbb{F}}$ with norm $n$. Then the Dedekind zeta function associated to  $\mathbb{F}$ is defined by the following Dirichlet series:
\begin{align}\label{Dedekind}
\zeta_\mathbb{F}(s) = \sum_{\mathfrak{a} \subset \mathcal{O}_{\mathbb{F}}} \frac{1}{\mathfrak{N}(\mathfrak{a})^s} = \sum_{n=1}^{\infty}\frac{\mathtt{a}_{\mathbb{F}}(n)}{n^s},  \quad \Re(s)>1,
\end{align}
where $\mathfrak{a}$ runs over the non-zero integral ideals of $\mathcal{O}_{\mathbb{F}}$. 
%\end{definition}
%Parallel to Riemann zeta function, $\zeta_{\mathbb{F}}$ also has Euler product expansion, given by
%$$ \zeta_\mathbb{F}(s) = \prod_{\mathfrak{p} \subset \mathcal{O}_\mathbb{F}}\left(1 - \frac{1}{\mathfrak{N}(\mathfrak{p})^s}\right)^{-1},$$
%where $\mathfrak{p}$ runs over prime ideals of $\mathcal{O}_{\mathbb{F}}$.
It is well known that $\zeta_{\mathbb{F}}(s)$ has a simple pole at $s=1$ and
the residue is given by the class number formula: 
\begin{align}\label{Class number formula}
\lim_{s \rightarrow 1}(s-1)\zeta_{\mathbb{F}}(s) = \frac{2^{r_1}(2\pi)^{r_2}}{\sqrt{D}}\frac{R_{\mathbb{F}}h_{\mathbb{F}}}{w_{\mathbb{F}} } :=H_{ \mathbb{F}},
\end{align}
where $w_{\mathbb{F}}$ denotes the number of roots of unity contained in $\mathcal{O}_\mathbb{F}, h_{\mathbb{F}}$ is the class number, and $R_{\mathbb{F}}$ denotes the regulator of $\mathbb{F}$.  Throughout the paper,  for simplicity,  we denote the above residue as $H_{ \mathbb{F}}$.  Further,  we know that $r=r_1+r_2-1$  is the rank of the unit group of $\mathbb{F}$ and $\zeta_{\mathbb{F}}(s)$ has a zero of order $r$ at $s=0$ with 
\begin{align}\label{Laurent series_at s=0_1st coeff}
\lim_{s\rightarrow 0} \frac{\zeta_{\mathbb{F}}(s)}{s^r} = - \frac{R_\mathbb{F} h_\mathbb{F}}{w_\mathbb{F}} :=C_{\mathbb{F}}.  
\end{align}
Note that the constants $H_{\mathbb{F}}$ and $C_{\mathbb{F}}$ are related by the relation $\sqrt{D} H_{\mathbb{F}}=- 2^{r_1}(2\pi)^{r_2} C_{\mathbb{F}}$.  
Now we define a \emph{generalized divisor function} attached to a given number field $\mathbb{F}$,  which was recently studied by Gupta and Pandit \cite[Equation (1.5)]{GP2021}:
\begin{align}\label{Gupta_Pandit}
\sigma_{\mathbb{F}, \ell}(n) = \sum_{d|n}\mathtt{a}_{\mathbb{F}}(d)\mathtt{a}_{\mathbb{F}}\left(\frac{n}{d}\right)d^\ell.
\end{align}
In their paper, they investigated Riesz sum associated to $\sigma_{\mathbb{F},\ell}(n)$. 
One can check that the Dirichlet series associated to $\sigma_{\mathbb{F},  \ell}(n)$ is given by
\begin{align}\label{Dirichlet series for generalized divisor function}
\sum_{n=1}^{\infty}\frac{\sigma_{\mathbb{F}, \ell}(n)}{n^s} = \zeta_{\mathbb{F}}(s)\zeta_{\mathbb{F}}(s- \ell), \quad \Re(s) > \max\{1, 1+\Re(\ell)\}.
\end{align}
A number field analogue of Euler's identity \eqref{Euler's formula for even zeta values} has been given by Klingen \cite{Klingen} and Siegel \cite{Siegel}.   It states that for any totally real number field $\mathbb{F}$ of degree $n$, we have
% which has discriminant $D$,
\begin{align}\label{Dedekind zeta at even numbers for real fields}
\zeta_{\mathbb{F}}(2m) = \frac{q_m \pi^{2mn}}{\sqrt{D}}, \quad m \in \mathbb{N},
\end{align}
where $q_m$ is some fixed non-zero rational number. From the above identity, we notice that like $\zeta(2m)$, even zeta values over totally real number fields are also transcendental. Recently,  Murty and Pathak \cite{MP2021} have also studied the arithmetic nature of Dedekind zeta function at odd positive integers. Their result states that for a given integer $n \geq 1$, at most one of $\zeta_{\mathbb{F}}(2n+1)$ is rational when $\mathbb{F}$ runs over all imaginary quadratic fields. 
Recently, Banerjee, Gupta and Kumar \cite{BGK23} have studied transformation formula for Dedekind zeta function at odd integers which gives a generalization of \eqref{Ramanujan's Formula}.   However,  our generalization is different from their result.  

%We have also generalized Ramanujan's formula for Dedekind zeta function, but in a different way. 

Now we discuss the Steen function before mentioning the main result.
%\begin{definition}{{\bf (Steen function)}}
The Steen function $V(z| a_1, a_2, \ldots, a_n)$ is defined by 
\begin{align}\label{Voronoi Steen function}
V(z| a_1, a_2, \ldots, a_n) = \frac{1}{2\pi i}\int_{(c)}\prod_{j=1}^{n}\Gamma(s + a_j)z^{-s}ds.
\end{align}
Here and throughout the article the symbol $(c)$ denotes the vertical line from $c - i\infty$ to $c + i\infty$.  We assume that all the poles of $\Gamma(s + a_j)$ lie on one side of the vertical line $(c)$.  One can check that this function is a particular case of Meijer $G$-function.  
%\end{definition}
%Vorono\"i function is a generalization of exponential function as $V(z; 0) = e^{-z}$. 
Special cases of Steen function are associated to many other well-known functions such as,  %One of them is given by
\begin{align}
V(z| 0) &= e^{-z},  ~~\textrm{if}~~ c>0,  \label{Voronoi-exp} \\
V(z| a, b) &= 2z^{\frac{1}{2}(a+b)}K_{a-b}(2z^{\frac{1}{2}}),  ~~\textrm{if}~~ c> \max\{-a,  -b\},\label{Bessel fn formula}
\end{align}
where $K_{\nu}$ denotes the modified Bessel function of the second kind.  Further information about Steen function can be found in \cite{Steen},  \cite[p.~63]{KT}.  

Recall that Ramanujan's formula \eqref{Ramanujan's Formula} as well as Grosswald's identity \eqref{Grosswald identity} has an infinite series \eqref{infinite series_Grosswald} containing the divisor function $\sigma_{z}(n)$ and the exponential function,  that is,  for $k \in \mathbb{Z},  z \in \mathbb{H}$,  
\begin{align} \label{series_Grosswald}
F_k(z) = \sum_{n=1}^{\infty}\sigma_{-k}(n)e^{2\pi inz}.
\end{align}
In the current paper, we are interested  to study transformation formula for the following infinite series:
\begin{align}\label{Infinite series F}
\mathfrak{F}_{\mathbb{F}, k}(z) := \sum_{n=1}^{\infty}  \sigma_{\mathbb{F}, -k}(n)V\left(-\frac{(2\pi)^d niz}{D} \bigg| \bar{0}_d \right).
\end{align}
%The above series contains the generalized forms namely, $\sigma_{\mathbb{F}, k}(n)$ and $V(z| \bar{0}_d)$, where $\bar{0}_d$ denotes the $d$-tuple with all entries equal to 1. One can easily check that for $\mathbb{F} = \mathbb{Q}$, we have $\mathfrak{F}_{\mathbb{Q}, k}(z) = F_k(z)$.
One can easily check that the above series \eqref{Infinite series F} reduces to \eqref{series_Grosswald} when $\mathbb{F} = \mathbb{Q}$.  Now we are ready to state the main results of this paper in the next section.

\section{Main Results}
%The main results of this paper are stated below.
\begin{theorem}\label{DB}
Let $\mathbb{F}$ be a number field of degree $d=r_1+2r_2$ and $\zeta_{\mathbb{F}}(s)$ be the Dedekind zeta function defined in \eqref{Dedekind}.  We consider $r =r_1+r_2-1$.  For any non-zero integer $k$,  
we define 
\begin{align}\label{Lamda_F_k}
\Lambda_{\mathbb{F}, k}(s) := \Gamma(s)^d\zeta_{\mathbb{F}}(s)\zeta_{\mathbb{F}}(s+2k+1)\left(\frac{(2\pi)^d}{D}\right)^{-s}.
\end{align}
Let $\mathfrak{F}_{\mathbb{F}, k}(z)$ be the infinite series defined as in \eqref{Infinite series F}. Then for any $z \in \mathbb{H}$ and 
 $k >0$,  we have 
\begin{align}
\mathfrak{S}_{\mathbb{F},  2k+1}(z) & = (-1)^{k(r_1+1)+r_2}z^{2k} \mathfrak{S}_{\mathbb{F},  2k+1} \left(- \frac{1}{z}\right)   
 +  {\displaystyle \sum_{j=1}^k  \mathfrak{R}_{-(2j-1)}(z) + \sum_{j=1}^{k-1} \mathfrak{R}_{-2j}(z)},  
\end{align}
where 
\begin{align*}
\mathfrak{S}_{\mathbb{F},  2k+1}(z):=  \mathfrak{F}_{\mathbb{F}, 2k+1}(z) - \mathfrak{R}_{0}(z)- \mathfrak{R}_1(z), 
\end{align*}
and the residual terms are defined as 
{\allowdisplaybreaks \begin{align*}
\mathfrak{R}_{0}(z) &  =  \frac{1}{(r_2)!}\lim_{s \rightarrow 0} \frac{{\rm d}^{r_2}}{{\rm d}s^{r_2}}\left(s^{r_2+1}    \Lambda_{\mathbb{F},k}(s) (- i z) ^{-s}\right),     \nonumber \\
 \mathfrak{R}_1(z)  & = H_{\mathbb{F} }  \zeta_{\mathbb{F}}(2k+2)  \frac{ i D  }{ (2\pi)^d z},  \\
   \mathfrak{R}_{-(2j-1)}(z) & = \frac{1}{(r)!}\lim_{s \rightarrow -(2j-1)} \frac{{\rm d}^{r}}{{\rm d}s^{r}}\left((s+2j-1)^{r+1}\Lambda_{\mathbb{F}, k}(s) (- iz)^{-s}\right),   \\
   \mathfrak{R}_{-2j}(z) & =  \frac{1}{(r_2-1)!}\lim_{s \rightarrow -2j} \frac{{\rm d}^{r_2-1}}{{\rm d}s^{r_2-1}}\left((s+2j)^{r_2} \Lambda_{\mathbb{F}, k}(s) (- i z)^{-s}\right).
\end{align*}}
Again,  for $k <0$,   we have 
\begin{align}\label{for k<0}
\mathfrak{U}_{\mathbb{F},  2k+1}(z) & = (-1)^{k(r_1+1)+r_2}z^{2k} \mathfrak{U}_{\mathbb{F},  2k+1} \left(- \frac{1}{z}\right)   + \mathfrak{R}(z),
\end{align}
where
\begin{align*}
\mathfrak{U}_{\mathbb{F},  2k+1}(z) :=  \mathfrak{F}_{\mathbb{F}, 2k+1}(z)  -    \frac{ C_{\mathbb{F}}  \zeta_{\mathbb{F}}^{(r_2)}(2k+1)}{(r_2)!},
\end{align*}
and 
 \begin{align}\label{Residue at 1 for k negative}
 \mathfrak{R}(z)=  \begin{cases}
 -\frac{i}{4\pi z},  & \quad \text{if} \, \, \, (k,r_1,  r_2) = (-1,1,  0), \\
  \frac{ H_{\mathbb{F} }  \zeta_{\mathbb{F}}(0) D i }{ (2\pi)^2 z} ,  & \quad \text{if} \, \, \, (k,r_1,  r_2) = (-1,0,  1), \\
0,  & \quad \text{otherwise.}
\end{cases}
\end{align}

\end{theorem}

\begin{remark}
We note that an explicit evaluation of the terms $\mathfrak{R}_0(z),  \mathfrak{R}_{-(2j-1)}(z)$ and $\mathfrak{R}_{-2j}(z)$ is not easy as it involves higher derivatives.  However,  we can say that $\mathfrak{R}_0(z)$ is a polynomial in $\mathbb{C}[\log(z)]$ of degree $r_2$,  whereas  $\mathfrak{R}_{-(2j-1)}(z)$ and $\mathfrak{R}_{-2j}(z)$ are polynomials of the form $z^{2j-1}g(\log(z))$ and $z^{2j} h(\log(z))$,  where $g(z)$ and $h(z)$ are some polynomials of degree $r=r_1+r_2-1$ and $r_2 -1$,  respectively.

% in $\mathbb{C}[z, \log(z)]$ of  degree $2j + r_1 + r_2 -2$ and $2j+r_2-1$,  respectively. 
\end{remark}

\begin{corollary}\label{When F equals Q}
For $\mathbb{F}=\mathbb{Q}$,  Theorem \ref{DB} reduces to  the Ramanujan-Grosswald formula \eqref{Grosswald identity}.
\end{corollary}

In the next subsection, we highlight some of the intriguing identities that are by product of our main result in the case of totally real number field.
\subsection{Transformation formulae for totally real number fields}

The next implication of Theorem \ref{DB} is stated as a separate theorem since  it can be regarded as a formula for $\zeta_{\mathbb{F}}(2k+1)$ over a totally real number field.     

%The following identity is obtained when we consider $\mathbb{F}$ to be a totally real number field with $k>0$,
\begin{theorem}\label{k>0 and real number field}
Let $\mathbb{F}$ be a totally real number field of degree $r_1$ and $k$ be a positive integer. Then we have
\begin{align}\label{k>0 and real number field main st}
\mathfrak{S}_{\mathbb{F},  2k+1}(z) & = (-1)^{k(r_1+1)}z^{2k} \mathfrak{S}_{\mathbb{F},  2k+1} \left(- \frac{1}{z}\right)   
 +  {\displaystyle \sum_{j=1}^k  \mathfrak{R}_{-(2j-1)}(z)} ,  
\end{align}
where 
\begin{align*}
\mathfrak{S}_{\mathbb{F},  2k+1}(z) =  \mathfrak{F}_{\mathbb{F}, 2k+1}(z) - C_\mathbb{F} \zeta_{\mathbb{F}}(2k+1) - H_{\mathbb{F} }  \zeta_{\mathbb{F}}(2k+2)  \frac{ i D  }{ (2\pi)^{r_1} z}, 
\end{align*}
and the residual term is given by
\begin{align*}
%\mathfrak{R}_{0}(z) &  =  \frac{\zeta_{\mathbb{F}}^{(r_1 - 1)}(0)}{(r_1 - 1)!} \zeta_{\mathbb{F}}(2k+1),\nonumber \\
%\mathfrak{R}_1(z)  & = H_{\mathbb{F} }  \zeta_{\mathbb{F}}(2k+2)  \frac{ i D  }{ (2\pi)^d z},  \\
\mathfrak{R}_{-(2j-1)}(z) & = \frac{1}{(r_1-1)!}  \lim_{s \rightarrow -(2j-1)} \frac{{\rm d}^{r_1-1}}{{\rm d}s^{r_1-1}}\left((s+2j-1)^{r_1}\Lambda_{\mathbb{F}, k}(s) (- iz)^{-s}\right).
\end{align*}
\end{theorem}

%\begin{remark}
%In particular,  when $\mathbb{F} = \mathbb{Q}$,   one can show that  Theorem {\rm \ref{k>0 and real number field}} yields Grosswald's identity \eqref{Grosswald identity} for $\zeta(2k+1)$.
%\end{remark}

In particular,  for the real quadratic fields we obtain the following identity. 
\begin{corollary}\label{k>0 and quadratic real field}
Let $k > 0,$ and $\mathbb{F}$ be a real quadratic field i.e. $\mathbb{F} = \mathbb{Q}(\sqrt{m})$ where m is a positive square free integer. Then  we have
\begin{align}\label{k>0 and quadratic real field eqn}
\mathfrak{S}_{\mathbb{F},  2k+1}(z) & = (-1)^{k}z^{2k} \mathfrak{S}_{\mathbb{F},  2k+1} \left(- \frac{1}{z}\right)   
 +  {\displaystyle \sum_{j=1}^k  \mathfrak{R}_{-(2j-1)}(z)} ,  
\end{align}
where
\begin{align*}
\mathfrak{S}_{\mathbb{F},  2k+1}(z) &= 2\sum_{n = 1}^{\infty}\sigma_{\mathbb{F}, -2k-1}(n) K_0\left( 2 \pi  \sqrt{\frac{nz}{m}} e^{- \frac{i \pi}{4}}\right)- \zeta_{\mathbb{F}}^{'}(0)\zeta_{\mathbb{F}}(2k+1) -   \frac{ H_{\mathbb{F} }  \zeta_{\mathbb{F}}(2k+2) i m }{ \pi^2 z}. 
\end{align*}
The term $\mathfrak{R}_{-(2j-1)}(z)$ in \eqref{k>0 and quadratic real field eqn} is given by 
\begin{align*}
\mathfrak{R}_{-(2j-1)}(z) & = \lim_{s \rightarrow -(2j-1)} \frac{{\rm d}}{{\rm d}s}\left((s+2j-1)^{2}\Lambda_{\mathbb{F}, k}(s) (- iz)^{-s}\right).
\end{align*}
%{\bf Can we simplify $ \mathfrak{R}_{-(2j-1)}(z)$ by writing $\zeta_{\mathbb{F}}(s)=\zeta(s) L(s,  \chi)$?}
\end{corollary}

Next,  we present identities for Dedekind zeta function over totally real number fields at negative odd integers.  
%We now state some subcorollaries of Corollary \ref{Corollary for k<0 for any field}.
\begin{theorem}\label{k<0 and totally real fields}
Let $k$ be a positive integer and $\mathbb{F}$ be a totally real number field of degree $r_1.$ Then,  we have
{\allowdisplaybreaks \begin{align*}
& z^{2k}\left\{  \mathfrak{F}_{\mathbb{F}, -2k+1}(z) - C_\mathbb{F} \zeta_{\mathbb{F}}(1-2k) \right\} \nonumber \\
& = (-1)^{k(r_1 + 1)}\left\{  \mathfrak{F}_{\mathbb{F}, -2k+1}\left(-\frac{1}{z}\right)
- C_\mathbb{F} \zeta_{\mathbb{F}}(1-2k) \right\} 
 + \begin{cases}
\frac{z^{2k}}{4\pi zi}, & \text{if} \,\, (k,r_1,  r_2) = (1,1,  0), \\
0, & \text{otherwise}.
\end{cases}
\end{align*} }

%\begin{align*}
% \mathfrak{F}_{\mathbb{F}, 2k+1}(z) - \frac{\zeta_{\mathbb{F}}^{(r_1 - 1)}(0)}{(r_1 - 1)!}\zeta_{\mathbb{F}}(2k+1) 
% & = (-1)^{k(r_1+1)}z^{2k} \left\{  \mathfrak{F}_{\mathbb{F}, 2k+1}\left(-\frac{1}{z}\right) - \frac{\zeta_{\mathbb{F}}^{(r_1 - 1)}(0)}{(r_1 - 1)!} \zeta_{\mathbb{F}}(2k+1) \right\} \nonumber \\
%& + \begin{cases}
% -\frac{i}{4\pi z},  & \quad \text{if} \, \, \, (k,r_1,  r_2) = (-1,1,  0), \\
%0,  & \quad \text{otherwise.}
%\end{cases}
%\end{align*}
\end{theorem}
%  - \frac{ H_{\mathbb{F} }  \zeta_{\mathbb{F}}(0) D i }{ (2\pi)^2 z} ,  & \quad \text{if} \, \, \, (k,r_1,  r_2) = (-1,0,  1), \\
%From above theorem, we can derive Grosswald's identity as follows:
\begin{remark}
When $(r_1,  r_2) = (1,0) $ and $k>1$, then the above formula immediately transforms into  the Grosswald's identity \eqref{Grosswald's id without finite sum}.  
%Further, if we replace $k $ by $-k$,  then for a real number field $\mathbb{F}$ with $k > 0$,  one has
%\begin{align*}
%& z^{2k}\left\{  \mathfrak{F}_{\mathbb{F}, -2k+1}(z) - \frac{\zeta_{\mathbb{F}}^{(r_1 - 1)}(0)}{(r_1 - 1)!}\zeta_{\mathbb{F}}(1-2k) \right\} \nonumber \\
%& = (-1)^{k(r_1 + 1)}\left\{  \mathfrak{F}_{\mathbb{F}, -2k+1}\left(-\frac{1}{z}\right)
%- \frac{\zeta_{\mathbb{F}}^{(r_1 - 1)}(0)}{(r_1 - 1)!} \zeta_{\mathbb{F}}(1-2k) \right\} 
% + \begin{cases}
%\frac{z^{2k}}{4\pi zi}, & \text{if} \,\, (k,r_1,  r_2) = (1,1,  0), \\
%0, & \text{otherwise}.
%\end{cases}
%\end{align*} 
%  \frac{ H_{\mathbb{F} }   \zeta_{\mathbb{F}}(0) D z^{2k} }{ (2\pi)^2 i z} ,  & \text{if} \, \, \, (k,r_1,  r_2) = (1,0,  1), \\
% Let $(k, r_1) = (1,1)$ in \eqref{Remark1}, so we have $\mathbb{F} = \mathbb{Q}, D = 1, d = 1$. 
In particular,  when $(k, r_1,  r_2) = (1,1,  0)$,  the above   formula  yields  the well known transformation formula for the Eisenstein series of weight $2$,  that is,  
\begin{align*}
E_2\left(-\frac{1}{z}\right) &= z^2\left(E_2(z) + \frac{6}{\pi iz}\right).
\end{align*}
\end{remark}
The next result gives an explicit evaluation of the infinite series $ \mathfrak{F}_{\mathbb{F}, -2k+1}(z)$ at $z=i$ that involves the class number and value of the Dedekind zeta function at negative odd integers. 
% involving the Dedekind zeta function at negative odd integers, obtained by putting $z=i$ in \eqref{Remark1},
\begin{corollary}\label{Evaluation at z=i}
Let $k \geq 3$ and $r_1 \geq 1$ be odd integers.  
Then for a totally real field $\mathbb{F}$ of degree $r_1$, we have 
\begin{align}\label{Exact evaluation_totally real_at i}
%\mathfrak{F}_{\mathbb{F}, -2k+1}(i) =
 \sum_{n=1}^{\infty}\sigma_{\mathbb{F}, 2k -1}(n)V\left(\frac{(2\pi)^{r_1} n}{D} \bigg| \bar{0}_{r_1} \right) = - \frac{h_\mathbb{F} R_\mathbb{F} \zeta_{\mathbb{F}}(1-2k)}{ w_\mathbb{F}}. 
\end{align}
This gives a  new identity for the class number of a totally real field.  Using functional equation \eqref{Dedekind_zeta_functional_equation} and Siegel's identity \eqref{Dedekind zeta at even numbers for real fields},  one can check that $\zeta_{\mathbb{F}}(1-2k)$ is a non-zero rational number. 
\end{corollary}

\begin{remark}
In particular,  when $\mathbb{F} = \mathbb{Q}$,   Corollary {\rm \ref{Evaluation at z=i}} gives an exact evaluation of the following well-known Lambert series
\begin{align*}
\sum_{n=1}^{\infty}\sigma_{2k-1}(n)e^{-2\pi n} &= - \frac{\zeta(1-2k)}{2} = \frac{B_{2k}}{4k}.
\end{align*}
\end{remark}
This identity was first obtained by Glaisher and later rediscovered by Ramanujan  \cite[p.~262,  Equation (13.1)]{BCB-II}.  
%\begin{corollary}\label{Transformation formula for E2z}
%Let $z \in \mathbb{H}$. Then we have following transformation formula for Eisenstein series of weight 2:
%\begin{align}
%E_2\left(-\frac{1}{z}\right) = z^2\left(E_2(z) + \frac{6}{\pi iz}\right).
%\end{align}
%
%\end{corollary}

Next,  we provide an interesting analogue of the transformation formula \eqref{Transformation2} for Eisenstien series $G_{2k}(z)$,  namely,  a transformation formula that sends $z \rightarrow -\frac{1}{z}$
   for real quadratic fields. 
   %that sends $z \rightarrow -\frac{1}{z}$.  

 %following generalization of the transformation formula \eqref{Transformation2} (relating $z \rightarrow \left(-\frac{1}{z}\right)$) for real quadratic number fields. 
\begin{corollary}\label{Transformation formula for quadratic fields}
Let $m$ be a positive square free integer and $\mathbb{F} = \mathbb{Q}(\sqrt{m})$.  
 %where m is a positive square free integer. 
 Then for $z \in \mathbb{H}$ and $k \in \mathbb{N}$,  we have
\begin{align*}
(iz)^{2k} G_{\mathbb{F}, 2k}(z) = G_{\mathbb{F}, 2k}\left( -\frac{1}{z}\right), 
\end{align*}
where
\begin{align*}
G_{\mathbb{F}, 2k}(z) := 1 - \frac{2}{\zeta_{\mathbb{F}}'(0)\zeta_{\mathbb{F}}(1-2k)}\sum_{n=1}^{\infty}\sigma_{\mathbb{F}, 2k-1}(n)K_0\left(2\pi \sqrt{\frac{nz}{m}} e^{-\frac{i\pi}{4}}\right).
\end{align*}
\end{corollary}
%A numerical verification of the above result is given in Table \ref{Table of main theorem}.  

Further,  as an application of Theorem \ref{DB}, we give number field analogue of Ramanujan-Grosswald identity for imaginary number fields.

\subsection{Transformation formulae for imaginary number fields}

\begin{theorem}\label{k>0 and imaginary field}
Let $\mathbb{F}$ be a purely imaginary number field with degree of extension $2 r_2$ over $\mathbb{Q}.$ Then for $k>0,  z \in \mathbb{H}$,  we have
\begin{align}\label{k>0 and imagianry field main eqn}
\mathfrak{S}_{\mathbb{F},  2k+1}(z) & = (-1)^{k+r_2}z^{2k} \mathfrak{S}_{\mathbb{F},  2k+1} \left(- \frac{1}{z}\right)   
 +  {\displaystyle \sum_{j=1}^k  \mathfrak{R}_{-(2j-1)}(z) + \sum_{j=1}^{k-1} \mathfrak{R}_{-2j}(z)},
\end{align}
where 
\begin{align*}
\mathfrak{S}_{\mathbb{F},  2k+1}(z)=  \mathfrak{F}_{\mathbb{F}, 2k+1}(z) - \mathfrak{R}_{0}(z)- \mathfrak{R}_1(z), 
\end{align*}
and the residual terms are defined as 
\begin{align*}
\mathfrak{R}_{0}(z) &  =  \frac{1}{(r_2)!}\lim_{s \rightarrow 0} \frac{{\rm d}^{r_2}}{{\rm d}s^{r_2}}\left(s^{r_2+1}    \Lambda_{\mathbb{F},k}(s) (- i z) ^{-s}\right),     \nonumber \\
 \mathfrak{R}_1(z)  & = H_{\mathbb{F} }  \zeta_{\mathbb{F}}(2k+2)  \frac{ i D  }{ (2\pi)^{2 r_2} z},  \\
   \mathfrak{R}_{-(2j-1)}(z) & = \frac{1}{(r_2-1)!}\lim_{s \rightarrow -(2j-1)} \frac{{\rm d}^{r_2-1}}{{\rm d}s^{r_2-1}}\left((s+2j-1)^{r_2}\Lambda_{\mathbb{F}, k}(s) (- iz)^{-s}\right),   \\
   \mathfrak{R}_{-2j}(z) & =  \frac{1}{(r_2-1)!}\lim_{s \rightarrow -2j} \frac{{\rm d}^{r_2-1}}{{\rm d}s^{r_2-1}}\left((s+2j)^{r_2} \Lambda_{\mathbb{F}, k}(s) (- i z)^{-s}\right).
\end{align*}
Again,  for $k <0$,   we have 
\begin{align}\label{for k<0_imaginary field}
\mathfrak{U}_{\mathbb{F},  2k+1}(z) & = (-1)^{k +r_2}z^{2k} \mathfrak{U}_{\mathbb{F},  2k+1} \left(- \frac{1}{z}\right)   + \mathfrak{R}(z),
\end{align}
where
\begin{align}\label{U-function for imag field}
\mathfrak{U}_{\mathbb{F},  2k+1}(z) :=  \mathfrak{F}_{\mathbb{F}, 2k+1}(z)  -   \frac{C_\mathbb{F} \zeta_{\mathbb{F}}^{(r_2)}(2k+1)}{(r_2)!},
\end{align}
and 
 \begin{align}\label{Residue at 1 for imag field}
 \mathfrak{R}(z)=  \begin{cases}
  \frac{ H_{\mathbb{F} }  \zeta_{\mathbb{F}}(0) D i }{ (2\pi)^2 z} ,  & \quad \text{if} \, \, \, (k,r_1,  r_2) = (-1,0,  1), \\
0,  & \quad \text{otherwise.}
\end{cases}
\end{align}

\end{theorem}

In particular,  for quadratic imaginary fields we obtain the following transformation formula. 

\begin{corollary}\label{k>0 and quadratic imaginary field}
Let $m$ be a positive square free integer and $\mathbb{F} = \mathbb{Q}(\sqrt{-m})$ be a quadratic imaginary number field. Then for $k>0$,  we have
\begin{align}\label{k>0 and quadratic imaginary field main eqn}
\mathfrak{S}_{\mathbb{F},  2k+1}(z) & = (-1)^{k+1}z^{2k} \mathfrak{S}_{\mathbb{F},  2k+1} \left(- \frac{1}{z}\right)   
 +  {\displaystyle \sum_{j=1}^k  \mathfrak{R}_{-(2j-1)}(z) + \sum_{j=1}^{k-1} \mathfrak{R}_{-2j}(z)} ,  
\end{align}
where
\begin{align*}
\mathfrak{S}_{\mathbb{F},  2k+1}(z) &= 2\sum_{n = 1}^{\infty}\sigma_{\mathbb{F}, -2k-1}(n)K_0\left(2\pi \sqrt{\frac{nz}{m} } e^{-\frac{i \pi}{4}}\right) - \mathfrak{R}_{0}(z) - \mathfrak{R}_1(z),  
\end{align*}
and the residual terms are defined as 
\begin{align*}
\mathfrak{R}_{0}(z) &  =  \lim_{s \rightarrow 0} \frac{{\rm d}^{2}}{{\rm d}s^{2}}\left(s^{2}    \Lambda_{\mathbb{F},k}(s) (- i z) ^{-s}\right),     \nonumber \\
 \mathfrak{R}_1(z)  & = H_{\mathbb{F} }  \zeta_{\mathbb{F}}(2k+2)  \frac{ i m  }{ \pi^2 z},  \\
   \mathfrak{R}_{-(2j-1)}(z) & = -\frac{\zeta_{\mathbb{F}}^{'}(1-2j)}{[(2j-1)!]^2}\zeta_{\mathbb{F}}(2k-2j+2)\left(\frac{\pi^2 iz}{m}\right)^{2j-1},   \\
   \mathfrak{R}_{-2j}(z) & =  \frac{\zeta_{\mathbb{F}}^{'}(-2j)}{[(2j)!]^2}\zeta_{\mathbb{F}}(2k-2j+1)\left(\frac{\pi^2 iz}{m}\right)^{2j}.
\end{align*}

\end{corollary}

The next result also gives another modular transformation property associated to quadratic imaginary fields. 

\begin{corollary}\label{For k=-1,  r_2=1}
Let $m$ be a square free positive integer and $\mathbb{F} = \mathbb{Q}(\sqrt{-m})$.  Then for $z \in \mathbb{H}$,  we have
\begin{align}\label{U at k=-1}
z^2 \mathfrak{U}_{\mathbb{F},  -1}(z) & = \mathfrak{U}_{\mathbb{F},  -1} \left(- \frac{1}{z}\right)   + \frac{H_{\mathbb{F}} \zeta_{\mathbb{F}}(0) i z m}{\pi^2},
\end{align}
where
\begin{align*}
\mathfrak{U}_{\mathbb{F},  -1}(z) =  \mathfrak{F}_{\mathbb{F}, -1}(z)  - \zeta_{\mathbb{F}}(0)  \zeta_{\mathbb{F}}^{'}(-1).  
\end{align*}
Moreover,  letting $z=i$ in \eqref{U at k=-1}, we obtain the following exact evaluation: 
\begin{align}\label{Exact evaluation at z=i}
\sum_{n=1}^{\infty}\sigma_{\mathbb{F}, 1}(n)  K_{0} \left( 2 \pi \sqrt{\frac{n}{m}}   \right) =  \frac{1}{2} \zeta_{\mathbb{F}}(0) \zeta'_{\mathbb{F}}(-1) +  \frac{H_{\mathbb{F}} \zeta_{\mathbb{F}}(0)  m}{4 \pi^2}.
\end{align}
%{\bf Verify this result in Mathematica for $\mathbb{F}= \mathbb{Q}(i)$.  Can we further simplify this identity for $\mathbb{F}= \mathbb{Q}(i)$?}
\end{corollary}
The below identity holds for any purely imaginary number field.  
\begin{corollary}\label{k<0 and imaginary field}
Let $\mathbb{F}$ be  a purely imaginary field of degree $2r_2$ and $k$ be a positive integer.  If $(k,  r_2) \neq (1,  1)$,  then we have
\begin{align}\label{modular relation for imag field}
z^{2k}  \mathfrak{U}_{\mathbb{F},  1-2k}(z)  = (-1)^{k+r_2} \mathfrak{U}_{\mathbb{F},  1-2k} \left(- \frac{1}{z}\right),  
\end{align}
where
$$
\mathfrak{U}_{\mathbb{F},  1-2k}(z) =  \mathfrak{F}_{\mathbb{F}, 1-2k}(z)  -   \frac{C_\mathbb{F}  \zeta_{\mathbb{F}}^{(r_2)}(1-2k)}{(r_2)!}. 
$$

\end{corollary}
\begin{remark}\label{when r_2=1}
%In particular,  when $\mathbb{F}=\mathbb{Q}(\sqrt{-m} )$ is a quadratic imaginary field.  
Moreover,  substituting $z=i$ in \eqref{modular relation for imag field} and considering $r_2$ as an odd  positive integer,  we obtain the following interesting exact evaluation:
\begin{align}\label{Exact evaluation for purely imag field}
%\mathfrak{F}_{\mathbb{F}, -2k+1}(i) = 
\sum_{n=1}^{\infty}\sigma_{\mathbb{F}, 2k -1}(n)V\left(\frac{(2\pi)^{2 r_2} n}{D} \bigg| \bar{0}_{2 r_2} \right)= - \frac{h_\mathbb{F} R_\mathbb{F} \zeta_{\mathbb{F}}^{(r_2)}(1-2k)}{w_\mathbb{F}(r_2)!}.  
\end{align}
This result is true for purely imaginary number fields and it can be considered as an analogue of Corollary \ref{Evaluation at z=i},  which is true only for real number fields.  
In particular,  when $\mathbb{F}=\mathbb{Q}(\sqrt{-m} )$,  then for $k \geq 3$,  we have
\begin{align}\label{Exact evaluation_quad_imag_i}
%\mathfrak{F}_{\mathbb{F}, -2k+1}(i) = 
\sum_{n=1}^{\infty}\sigma_{\mathbb{F}, 2k -1}(n)  K_{0} \left( 2 \pi \sqrt{\frac{n}{m}}   \right) =   \frac{1}{2}\zeta_{\mathbb{F}}(0) \zeta'_{\mathbb{F}}(1-2k). 
\end{align}
%{\bf Verify this result in Mathematica for $\mathbb{F}= \mathbb{Q}(i)$.}
\end{remark}

%\begin{remark}
%{\bf We have to add a corollary for $(k,r_1,  r_2) = (-1,0,  1)$.}  This will be one of the cases for quardratic imaginary field. 
%\end{remark}

\subsection{A number field analogue of transformation formula for Dedekind eta function}

One of the key observations is that our main result,  i.e.,  Theorem \ref{DB} loses its validity for  $k=0$ as $\zeta_{\mathbb{F}}(2k+1)$ exhibits a simple pole at $1$.  Therefore,  corresponding $k=0$, we must handle it separately.  Quite interestingly,  in this case,  we will derive a  number field analogue of transformation formula for Dedekind eta function $\eta(z)$,  which is a half integral weight modular form.
%Therefore, we provide the following generalization of transformation formula for Dedekind eta function $\eta(z)$ corresponding to $k=0$.
\begin{theorem}\label{k=0 case}
Let $\mathbb{F}$ be a number field of degree $d$ and  $\mathfrak{F}_{\mathbb{F}, 1}(z)$ be the infinite series defined as in \eqref{Infinite series F}.  Then we have
\begin{align*}
\mathfrak{T}_{\mathbb{F}}(z) & = (-1)^{r_2} \mathfrak{T}_{\mathbb{F}} \left(-\frac{1}{z}\right) + \mathfrak{R}_{0}(z),  
\end{align*}
where 
\begin{align*}
\mathfrak{T}_{\mathbb{F}}(z):=  \mathfrak{F}_{\mathbb{F}, 1}(z) - \mathfrak{R}_1(z), 
\end{align*}
with 
\begin{align*}
\mathfrak{R}_{0}(z)  = \frac{1}{(r_2 + 1)!}{\displaystyle \lim_{s \rightarrow 0}} \frac{{\rm d}^{r_2 + 1}}{{\rm d}s^{r_2 + 1}}\left(s^{r_2+2}\Lambda_{\mathbb{F}, 0}(s)(-iz)^{-s}\right),  \quad
\mathfrak{R}_1(z) = H_{\mathbb{F}} \zeta_{\mathbb{F}}(2) \frac{i D}{(2\pi)^{d}z}.
\end{align*}
\end{theorem}
In particular,   we have the following modular transformation property for totally real number fields.  
\begin{corollary}\label{k=0 totally real field}
Let $\mathbb{F}$ be a totally real number field of degree $r_1$ and  $\mathfrak{F}_{\mathbb{F}, 1}(z)$ be the infinite series defined in \eqref{Infinite series F}.  Let $H_{\mathbb{F}},  \gamma_\mathbb{F}$ be the constants defined in \eqref{Laurent_1}. 
Then we have 
\begin{align}
\mathfrak{F}_{\mathbb{F}, 1}(z) - \mathfrak{F}_{\mathbb{F}, 1}\left(-\frac{1}{z}\right)& = a_0 \gamma_{\mathbb{F}} - a_0 H_{\mathbb{F}} \left\{ r_1 \gamma + \log\left(- \frac{(2 \pi)^{r_1} i z}{D}  \right) \right\}  + a_1 H_{\mathbb{F}} \nonumber   \\
&  + \mathfrak{R}_1(z) - \mathfrak{R}_1 \left(- \frac{1}{z} \right),  \label{totally real field_k=0}
\end{align}
where $$\mathfrak{R}_1(z) =  \frac{i D H_{\mathbb{F}} \zeta_{\mathbb{F}}(2)}{(2\pi)^{r_1}z}, \quad a_0 = \frac{\zeta_{\mathbb{F}}^{(r_1-1)}(0)  }{(r_1-1)!}=C_\mathbb{F}, \quad a_1= \frac{\zeta_{\mathbb{F}}^{(r_1)}(0)  }{(r_1)!}. $$
%{\bf Add Laurent series expansion of $\zeta_{\mathbb{F}}(s)$.  }
\end{corollary}

\begin{remark}
Quite surprisingly,  substituting $z=i$ in \eqref{totally real field_k=0} and using the relation between $H_\mathbb{F}$ and $C_\mathbb{F}$,  we obtain a connection between class number of a totally real number field of degree $r_1$ and the constant term $\gamma_{\mathbb{F}}$ of the Laurent series expansion \eqref{Laurent_1}  of the Dedekind zeta function at $s=1$.  Mainly,  we obtain
\begin{align}\label{another formula for class number}
C_\mathbb{F}=  \frac{ a_1-A \gamma_{\mathbb{F}}  }{  r_1 \gamma + \log\left( \frac{(2 \pi)^{r_1} }{D}  \right) },
\end{align}
where $A= \frac{\sqrt{D}}{2^{r_1} (2\pi)^{r_2}}$.  This suggests that finding a fomula for class number is also connected with the Kronceker's limit formula for the Dedekind zeta function.  

%{\bf Can we do some numerical verification of this identity.  Can we verify for $\mathbb{Q}(i)$.}

%{\bf This gives a new formula for class number a totally real number field.  Can we say that class number $H_{\mathbb{F}}$ will tend to infinity? }
\end{remark}
Further,  letting $\mathbb{F} = \mathbb{Q}$ in Corollary \ref{k=0 totally real field}, we have the following result.
\begin{corollary}\label{eta identity}
For any complex number $z \in \mathbb{H}$, we have
\begin{align}\label{Equivalent eta identity}
%\mathfrak{F}_{\mathbb{Q}, 1}(z) - \mathfrak{F}_{\mathbb{Q}, 1}\left(-\frac{1}{z}\right) = \frac{i \pi z^2 + i\pi}{12z} + \frac{1}{2}\log(-i z),  \\
\sum_{n=1}^{\infty} \frac{\sigma(n)}{n} e^{2\pi i n z} -  \sum_{n=1}^{\infty} \frac{\sigma(n)}{n} e^{- \frac{2\pi i n }{z}} = \frac{i \pi z^2 + i\pi}{12z} + \frac{1}{2}\log(-i z), 
\end{align}
which is equivalent to the logarithm of the transformation formula for the Dedekind eta function $\eta(z)$, namely,
\begin{align*}
\eta\left(- \frac{1}{z} \right)= \sqrt{-i z} \, \eta(z). 
\end{align*}
\end{corollary}

%{\bf Can we obtain some corollary for totally real number fields,  that is,  $d=r_1$ and $r_2=0$?}\\
%{\bf For imaginary quadratic fields? When $r_2 \geq 0$ even and $z=i$,  then $\mathfrak{R}_{0}(i)=0$ what does this mean?}

\section{Required Tools}
In this section,  we state a few essential results which will be frequently used in the proof of the main results.
%\begin{definition}{{\bf (Gamma function)}}
%The Gamma function is defined by the following convergent improper integral:
%$$\Gamma(z) = \int_0^\infty e^{-t}t^{s-1}dt, \;\;\;\;\; \Re(z) > 0.$$ 
%\end{definition}
%First we discusses the analyticity of $\Gamma(z)$ and the functional equation that it satisfies. 
The gamma function $\Gamma(z)$ satisfies the following functional equation and reflection formula:
\begin{align}
\Gamma(z+1) &= z\Gamma(z),  \quad z\in \mathbb{C}, \label{Functional eqn for Gamma} \\
\Gamma(z)\Gamma(1-z) &= \frac{\pi}{\sin\pi z},  \quad z \in \mathbb{C}\backslash \mathbb{Z}\label{Euler's Reflection Formula} \\
\Gamma(z)\Gamma\left(z + \frac{1}{2}\right)  &= 2^{1-2z}\sqrt{\pi}\Gamma(2z).  \label{Legendre's Duplication Formula}
\end{align}
%where $s \notin \mathbb{Z}.$ 
%\item Legendre's duplication formula:
%\begin{align}\label{Legendre's Duplication Formula}
%\Gamma(z)\Gamma\left(z + \frac{1}{2}\right) = 2^{1-2z}\sqrt{\pi}\Gamma(2z).
%\end{align} 

We now discuss analytic properties of the Dedekind zeta function $\zeta_{\mathbb{F}}(s)$.

\begin{lemma}
The Dedekind zeta function $\zeta_{\mathbb{F}}(s)$ has an analytic continuation in the whole complex plane except for a simple pole at $s=1$. It also satisfies the following functional equation relating its values at $s$ and $1-s$,  
\begin{align}\label{Dedekind_zeta_functional_equation}
\Omega_{\mathbb{F}}(s) = \Omega_{\mathbb{F}}(1-s),
\end{align}
where $\Omega_{\mathbb{F}}(s) = \left( \frac{D}{\pi^d 4^{r_2}} \right)^{\frac{s}{2}}  \Gamma\left(\frac{s}{2} \right)^{r_1}  \Gamma(s)^{r_2}\zeta_{\mathbb{F}}(s)$.
\end{lemma}

The Dedekind zeta function $\zeta_{\mathbb{F}}(s)$ satisfies the following Laurent series expansion at $s=1$:
\begin{align}\label{Laurent_1}
\zeta_{\mathbb{F}}(s)=  \frac{H_{\mathbb{F}}}{s-1} + \gamma_{\mathbb{F}} + O(s-1).  
\end{align}

Utilizing functional equation \eqref{Dedekind_zeta_functional_equation} of $\zeta_{\mathbb{F}}(s)$,  one can verify that it has a zero at $s=0$ of order $r=r_1 +r_2-1$.  
Therefore,  the following Laurent series expansion at $s=0$ holds:
\begin{align}\label{Laurent_0}
\zeta_{\mathbb{F}}(s) = a_0 s^{r_1 +r_2-1} + a_1 s^{r_1 +r_2} + O\left( s^{r_1 +r_2+1} \right),  
\end{align}
where $a_0 =  \frac{\zeta^{(r_1+r_2-1)}(0)}{(r_1+r_2-1)!}$ and $a_1 =  \frac{\zeta^{(r_1+r_2)}(0)}{(r_1+r_2)!}$.  From \eqref{Laurent series_at s=0_1st coeff},  one can observe that $a_0$ is nothing but the constant $C_\mathbb{F}$.  

%The residue of Dedekind zeta function at pole $s=1$ is given by the {\bf \emph{Class number formula}}: 
%\begin{align}\label{Class number formula}
%\lim_{s \rightarrow 0}(s-1)\zeta_{\mathbb{F}}(s) = \frac{2^{r_1}(2\pi)^{r_2}R_{\mathbb{F}}h_{\mathbb{F}}}{w_{\mathbb{F}} \sqrt{|D|}},
%\end{align}
%where $w_{\mathbb{F}}$ denotes the number of roots of unity contained in $\mathbb{F}, h_{\mathbb{F}}$ is the class number, and $R_{\mathbb{F}}$ denotes the regulator of $\mathbb{F}$.
Now we mention an interesting identity due to Chandrasekharan and Narasimhan \cite[Equation (57)]{CN61}.  For any integer $k$,  one can prove that
\begin{align*}
\Lambda_{k}(s)= (-1)^k \Lambda_{k}(-s-2k), 
\end{align*}
where $\Lambda_k(s)= (2 \pi)^{-s} \Gamma(s) \zeta(s) \zeta(s+2k+1).$
%\begin{align}
%\Gamma(s)\zeta(s)\zeta(s+2k+1) = (-1)^k(2\pi)^{2s+2k}\Gamma(-s-2k)\zeta(1-s)\zeta(-s-2k).
%\end{align}
We generalize this identity for Dedekind zeta function $\zeta_{\mathbb{F}}(s)$.  Mainly,  we prove the below result,  which will play a crucial role in proving our main identity.  
\begin{lemma}\label{Functional Equation for Product of Zeta Functions}
Let $\mathbb{F}$ be a number field of degree $d$ and $\zeta_{\mathbb{F}}(s)$ be the Dededkind zeta function defined in \eqref{Dedekind}.  For any $k \in \mathbb{Z}$, we have
\begin{align*}
\Lambda_{\mathbb{F}, k}(s) = (-1)^{kr_1 + r_2} \Lambda_{\mathbb{F}, k}(-s-2k),
\end{align*}
where $\Lambda_{\mathbb{F}, k}(s) = D^s(2 \pi)^{-ds} \Gamma(s)^d\zeta_{\mathbb{F}}(s)\zeta_{\mathbb{F}}(s+2k+1).$
\end{lemma}
\begin{proof}
We first prove this identity for any non-negative integer $k$. 
From the functional equation \eqref{Dedekind_zeta_functional_equation} of $\zeta_{\mathbb{F}}(s)$, we have
\begin{align}\label{Function_Eqn_for_Dedekind_Zeta_Fn}
&   \left( \frac{D}{\pi^d 4^{r_2}} \right)^{\frac{s}{2}}  \Gamma\left(\frac{s}{2}\right)^{r_1}\Gamma(s)^{r_2}\zeta_{\mathbb{F}}(s) =  \left( \frac{D}{\pi^d 4^{r_2}} \right)^{\frac{1-s}{2}}  \Gamma\left(\frac{1-s}{2}\right)^{r_1}\Gamma(1-s)^{r_2}\zeta_{\mathbb{F}}(1-s) .
\end{align}
Replace $s$ by $s + 2k + 1$ in \eqref{Function_Eqn_for_Dedekind_Zeta_Fn} to see that
\begin{align}\label{Function_Eqn_for_Dedekind_Zeta_Fn_(s+2k+1)}
 & \left( \frac{D}{\pi^d 4^{r_2}} \right)^{\frac{s+2k+1}{2}}  \Gamma\left(\frac{s+2k+1}{2}\right)^{r_1}\Gamma(s+2k+1)^{r_2}\zeta_{\mathbb{F}}(s+2k+1) \nonumber \\
 & =  \left( \frac{D}{\pi^d 4^{r_2}} \right)^{\frac{-s-2k}{2}  }  \Gamma\left( -\frac{s}{2}-k\right)^{r_1}\Gamma(-s-2k)^{r_2}\zeta_{\mathbb{F}}(-s-2k).
\end{align}
Now we multiply \eqref{Function_Eqn_for_Dedekind_Zeta_Fn} and \eqref{Function_Eqn_for_Dedekind_Zeta_Fn_(s+2k+1)} to get
{\allowdisplaybreaks \begin{align}\label{Product of two fnl eqn}
& \left( \frac{D}{\pi^d 4^{r_2}} \right)^{2s+2k} \left\{ \Gamma\left(\frac{s}{2}\right)  \Gamma\left(\frac{s+1}{2}+k\right) \right\}^{r_1}  \left\{ \Gamma(s)\Gamma(s+2k+1)  \right\}^{r_2}  \zeta_{\mathbb{F}}(s)\zeta_{\mathbb{F}}(s+2k+1) \nonumber \\
& =     \left\{ \Gamma\left(\frac{1}{2}- \frac{s}{2} \right)  \Gamma\left(-\frac{s}{2}-k \right) \right\}^{r_1}   \left\{ \Gamma(1-s) \Gamma(-s-2k)  \right\}^{r_2}   \zeta_{\mathbb{F}}(1-s)\zeta_{\mathbb{F}}(-s-2k). 
\end{align}}
%To simplifythe  above equation further, 
%we need to bring \eqref{Functional eqn for Gamma}, which gives, {\bf This is true for only $k\geq 0$.  What about negative integers?} 
To simplify the  above equation further,  we use the functional equation \eqref{Functional eqn for Gamma} of $\Gamma(s)$ repeatedly and find that,  for any non-negative integer $k$,   
\begin{align}\label{Gamma1}
\Gamma\left(\frac{s+1}{2}+ k\right) = \frac{1}{2^k}(s+2k-1)(s+2k-3) \cdots (s+1)\Gamma\left( \frac{s}{2} + \frac{1}{2}\right).
\end{align}
Again,  using \eqref{Functional eqn for Gamma} repeatedly,  one can show that
\begin{align}\label{Gamma2}
\Gamma\left(-\frac{s}{2}-k\right) = \frac{(-2)^k \Gamma\left(-\frac{s}{2}\right)}{(s+2k)(s+2k-2)\cdots (s+2)}.
\end{align}
Utilize \eqref{Gamma1},  \eqref{Gamma2} and duplication formula \eqref{Legendre's Duplication Formula} for $\Gamma(s)$ to derive
\begin{align}
\Gamma\left(\frac{s}{2}\right)\Gamma\left( \frac{s+1}{2} + k\right) 
&= \frac{2}{2^{s+k}}\sqrt{\pi}(s+2k-1)(s+2k-3)\cdots (s+1)\Gamma(s),  \label{Left Gamma1} \\
\Gamma\left(\frac{1}{2}- \frac{s}{2} \right)  \Gamma\left(-\frac{s}{2}-k \right)
&= \frac{(-1)^k 2^{1+s+k}\sqrt{\pi}\Gamma(-s)}{(s+2k)(s+2k-2) \cdots (s+2)}. \label{Left Gamma2}
\end{align}
We again employ \eqref{Functional eqn for Gamma} to see that,  for any non-negative integer $k$, 
%\begin{align}
%\Gamma(s+2k+1) & = (s+2k)(s+2k-1) \cdots s \Gamma(s),  \label{Gamma3} \\
%\Gamma(-s-2k) &= \frac{\Gamma(-s)}{(s+2k)(s+2k-1)\cdots(s+1)}.  \label{Gamma4}
%\end{align}
%Using \eqref{Gamma3}, \eqref{Gamma4} and \eqref{Functional eqn for Gamma}, we now get
\begin{align}
\Gamma(s)\Gamma(s+2k+1) &= s (s+1) \cdots (s+2k) \Gamma(s)^2,  \label{Right Gamma1} \\
\Gamma(-s-2k)\Gamma(1-s) 
&= \frac{-s(\Gamma(-s))^2}{(s+2k)(s+2k-1)\cdots (s+1)}.  \label{Right Gamma2}
\end{align}
Now substituting expressions from \eqref{Left Gamma1}-\eqref{Right Gamma2} in \eqref{Product of two fnl eqn} and upon simplification,  we have 
\begin{align}
& \left( \frac{D}{(2 \pi)^d} \right)^{2s+2k} \left\{ (s+1) (s+2) \cdots (s+2k) \right\}^{d} \Gamma(s)^d \zeta_{\mathbb{F}}(s) \zeta_{\mathbb{F}}(s+2k+1)  \nonumber \\
& = (-1)^{k r_1 +r_2} (\Gamma(-s))^d  \zeta_{\mathbb{F}}(1-s)\zeta_{\mathbb{F}}(-s-2k).  \label{Semifinal symmetric eqn}
\end{align}
Here we have used the fact that $r_1+2 r_2=d$. 
%Simplifying above equation, we get
%\begin{align}\label{Semifinal symmetric eqn}
%& D^s (2\pi)^{-(r_1 + 2r_2)s}(\Gamma(s))^{(r_1+2r_2)}\zeta_{\mathbb{F}}(s)\zeta_{\mathbb{F}}(s+2k+1) \nonumber \\
%& = \frac{D^{(-s-2k)}(2\pi)^{(s+2k)(r_1 + 2r_2)}(-1)^{(kr_1 + r_2)}\Gamma(-s)^{r_1+2r_2}}{[(s+2k)(s+2k-1)\ldots(s+2)(s+1)]^{r_1+2r_2}}\zeta_{\mathbb{F}}(1-s)\zeta_{\mathbb{F}}(-s-2k).
%\end{align}
Applying \eqref{Functional eqn for Gamma} again,  one can check that
\begin{align}\label{Final Gamma}
\frac{\Gamma(-s)}{(s+2k)(s+2k-1)\ldots(s+2)(s+1)} = \Gamma(-s-2k).
\end{align}
Using \eqref{Final Gamma}  in \eqref{Semifinal symmetric eqn}, we obtain
\begin{align*}
& \left( \frac{D}{(2 \pi)^d} \right)^{2s+2k} \Gamma(s)^d \zeta_{\mathbb{F}}(s) \zeta_{\mathbb{F}}(s+2k+1)  \nonumber \\
& = (-1)^{k r_1 +r_2} \Gamma(-s-2k)^d  \zeta_{\mathbb{F}}(1-s)\zeta_{\mathbb{F}}(-s-2k). 
\end{align*}
Now letting $\Lambda_{\mathbb{F}, k}(s) = D^s(2 \pi)^{-ds} \Gamma(s)^d\zeta_{\mathbb{F}}(s)\zeta_{\mathbb{F}}(s+2k+1)$,  the above identity reduces to $$\Lambda_{\mathbb{F}, k}(s)= (-1)^{k r_1 +r_2} \Lambda_{\mathbb{F}, k}(-s-2k).   $$
This completes the proof of this lemma for any non-negative integer $k$.   In a similar fashion,  one can prove that this identity also holds for any negative integer as well.  
\end{proof}

In the next section,  we present proofs of our main results. 
\section{Proof of Main Results}

\begin{proof}[Theorem \rm{\ref{DB}}][]
Using the definition \eqref{Voronoi Steen function} of the Steen function,  for $y>0$,  we write
\begin{align}\label{1st_Integral}
\sum_{n=1}^{\infty}\sigma_{\mathbb{F}, -2k -1}(n)V\left(\frac{(2\pi)^d ny}{D} \bigg| \bar{0}_d \right) &= \sum_{n=1}^{\infty}\sigma_{\mathbb{F}, -2k -1}(n)\frac{1}{2\pi i}\int_{(c)} \Gamma(s)^d \left(\frac{(2\pi)^d ny}{D}\right)^{-s} {\rm d}s \nonumber \\
&= \frac{1}{2\pi i}\int_{(c)} \Gamma(s)^d \sum_{n=1}^{\infty}\frac{\sigma_{\mathbb{F}, -2k -1}(n)}{n^s}\left(\frac{(2\pi)^d y}{D}\right)^{-s}{\rm d}s,
\end{align}
where we choose the line of integration as $  \max\{1, -2k\} <  \Re(s)= c < \max\{1, -2k\} + \epsilon$,  with $0< \epsilon<1$,  so that the above Dirichlet series converges absolutely and uniformly.   Hence,  the interchange of summation and integration is justifiable.  Now utilize the definition of the Dirichlet series \eqref{Dirichlet series for generalized divisor function} in \eqref{1st_Integral},  to see that
\begin{align}\label{sum_to_integral}
\sum_{n=1}^{\infty}\sigma_{\mathbb{F}, -2k -1}(n)V\left(\frac{(2\pi)^d ny}{D} \bigg| \bar{0}_d \right) &= \frac{1}{2\pi i}\int_{(c)} \Gamma(s)^d \zeta_{\mathbb{F}}(s)\zeta_{\mathbb{F}}(s+2k+1)\left(\frac{(2\pi)^d y}{D}\right)^{-s} {\rm d}s.
\end{align}
Now our aim is to study the following vertical line integral,   
\begin{align*} 
 J_{\mathbb{F},  k}(y) = \frac{1}{2\pi i}\int_{(c)} \Gamma(s)^d \zeta_{\mathbb{F}}(s)\zeta_{\mathbb{F}}(s+2k+1)\left(\frac{(2\pi)^d y}{D}\right)^{-s} {\rm d}s.    
 \end{align*} 
From the definition \eqref{Lamda_F_k} of  $\Lambda_{\mathbb{F}, k}(s)$,  it is clear that the above integral can be rewritten as
\begin{align*}
J_{\mathbb{F},  k}(y) =  \frac{1}{2\pi i}\int_{(c)}   \Lambda_{\mathbb{F}, k}(s) y^{-s} {\rm d}s.    
\end{align*}
We now set up a rectangular contour $\mathcal{C}$ made up of the vertices $c-iT, c+iT, \alpha+iT$ and $\alpha-iT$ taken in counter-clockwise direction. We already have $c > \max\{1, -2k\}$ with $T$ being some large positive quantity and we wisely choose $$\min\{-1,  -2k-2\}<\alpha < \min\{0, -2k-1\}.$$ 
% so that all poles of $\Lambda_{\mathbb{F},  k}(s)$  are inside the contour $\mathcal{C}$.
Before applying Cauchy's residue theorem,  we need to examine all the poles of $\Lambda_{\mathbb{F},  k}(s)$ inside the contour $\mathcal{C}$.   We know that $\Gamma(s)^d$ has poles of order $d$ at non-positive integers. 
Other poles will depend upon the value of $k$ which will affect location of vertices of $\mathcal{C}$.  Here we divide in two cases.  \\
\textbf{Case I)} Suppose $k > 0$.  
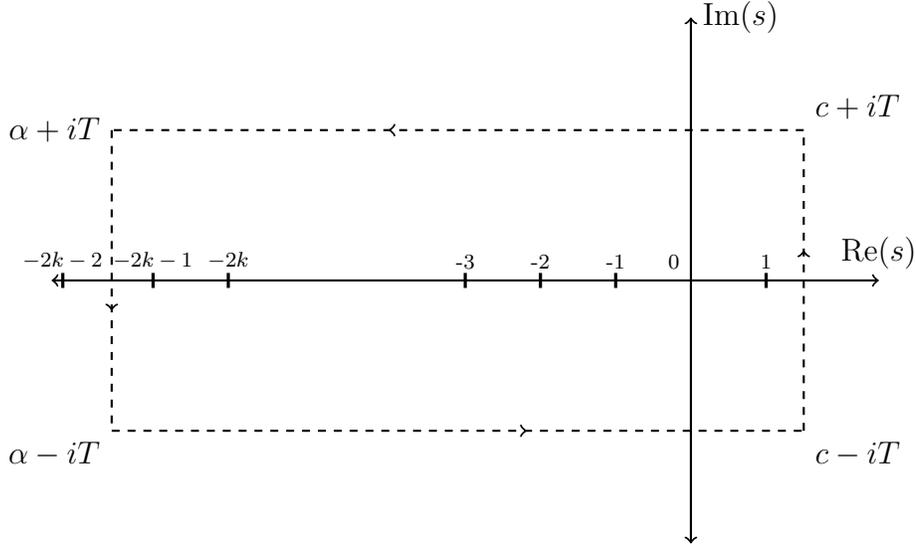
{\begin{center}
\begin{tikzpicture}[very thick,decoration={
  markings,
  mark=at position 0.6 with {\arrow{>}}}
 ] 
 \draw[thick,dashed,postaction={decorate}] (-4.2,-2)--(5,-2) node[below right, black] {$c-i T$};
 \draw[thick,dashed,postaction={decorate}] (5,-2)--(5,2) node[above right, black] {$c+iT$} ;
 \draw[thick,dashed,postaction={decorate}] (5,2)--(-4.2,2) node[left, black] {$\alpha+i T$}; 
 \draw[thick,dashed,postaction={decorate}] (-4.2,2)--(-4.2,-2) node[below left, black] {$\alpha-i T$}; 
 \draw[thick, <->] (-5,0) -- (6,0) coordinate (xaxis);
 \draw[thick, <->] (3.5,-3.5) -- (3.5,3.5)node[midway, above left, black] {\tiny0} coordinate (yaxis);
 \draw (2.5,0.1)--(2.5,-0.1) node[midway, above, black]{\tiny-1};
 \draw (1.5,0.1)--(1.5,-0.1) node[midway, above, black]{\tiny-2};
 \draw (0.5,0.1)--(0.5,-0.1) node[midway, above, black]{\tiny-3};
 \draw (-2.65,0.1)--(-2.65,-0.1) node[midway, above, black]{\tiny $-2k$} ;
 \draw (4.5,0.1)--(4.5,-0.1) node[midway, above, black]{\tiny1};
 \draw (-3.65,0.1)--(-3.65,-0.1) node[midway, above, black]{\tiny $-2k-1$};
 \draw (-4.85,0.1)--(-4.85,-0.1) node[midway, above, black]{\tiny $-2k-2$};
 \node[above] at (xaxis) {$\mathrm{Re}(s)$};
 \node[right] at (yaxis) {$\mathrm{Im}(s)$};

\end{tikzpicture}
\end{center}}
Now inside the contour $\mathcal{C}$,  for $\zeta_{\mathbb{F}}(s),\, s= 0$ is a zero of order $r_1+r_2 -1, \, s\in \{-2, -4, \ldots, -2k\}$ are zeros of order $r_1+r_2,\, s\in \{-1, -3, \ldots, -2k-1\}$ are zeros of order $r_2$ and $s=1$ is a simple pole. For $\zeta_{\mathbb{F}}(s+2k+1), s= -2k-1$ is a zero of order $r_1+r_2 -1$ and $s=-2k$ is a simple pole.
Hence final poles of  $\Lambda_{\mathbb{F}, k}(s)$  are at  $s = 0, -2k$  of order $r_2+1, s = 1,\, -2k-1$ are simple poles, $s\in \{-2, -4, \ldots, -2k+2\}$ are poles of order $r_2$ and $s\in \{-1, -3, \ldots, -2k+1\}$ are poles of order $r_1 + r_2$.\\
\textbf{Case II)} When $k < 0$. 

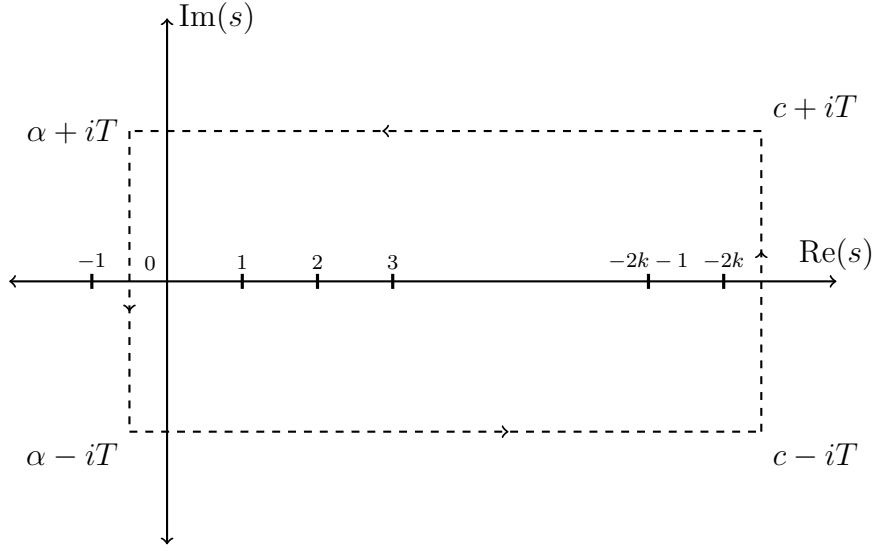
{\begin{center}
\begin{tikzpicture}[very thick,decoration={
  markings,
  mark=at position 0.6 with {\arrow{>}}}
 ] 
 \draw[thick,dashed,postaction={decorate}] (-3.4,-2)--(5,-2) node[below right, black] {$c-i T$};
 \draw[thick,dashed,postaction={decorate}] (5,-2)--(5,2) node[above right, black] {$c+iT$} ;
 \draw[thick,dashed,postaction={decorate}] (5,2)--(-3.4,2) node[left, black] {$\alpha+i T$}; 
 \draw[thick,dashed,postaction={decorate}] (-3.4,2)--(-3.4,-2) node[below left, black] {$\alpha-i T$}; 
 \draw[thick, <->] (-5,0) -- (6,0) coordinate (xaxis);
 \draw[thick, <->] (-2.9,-3.5) -- (-2.9,3.5)node[midway, above left, black] {\tiny0} coordinate (yaxis);
 \draw (-1.9,0.1)--(-1.9,-0.1) node[midway, above, black]{\tiny1};
 \draw (-0.9,0.1)--(-0.9,-0.1) node[midway, above, black]{\tiny2};
 \draw (0.1,0.1)--(0.1,-0.1) node[midway, above, black]{\tiny3};
 \draw (4.5,0.1)--(4.5,-0.1) node[midway, above, black]{\tiny $-2k$};
 \draw (3.5,0.1)--(3.5,-0.1) node[midway, above, black]{\tiny $-2k-1$};
 \draw (-3.9,0.1)--(-3.9,-0.1) node[midway, above, black]{\tiny $-1$};
 \node[above] at (xaxis) {$\mathrm{Re}(s)$};
 \node[right] at (yaxis) {$\mathrm{Im}(s)$};

\end{tikzpicture}
\end{center}}
We see that $s= 0$ is a zero of order $r_1+r_2 -1$ for $\zeta_{\mathbb{F}}(s)$ and a zero of order $r_2$ for $\zeta_{\mathbb{F}}(s+2k+1)$.  Note that $s=1$ is a simple pole for $\zeta_{\mathbb{F}}(s)$ and $s=-2k$ is a simple pole for $\zeta_{\mathbb{F}}(s+2k+1)$. 
Hence final poles for $\Lambda_{\mathbb{F}, k}(s)$, when $k < 0$,  are at  $s = 0, \, 1,\, -2k$,  which are all simple poles.

%\textbf{Case III)} When $k = 0$: $s=0$ is a pole of order $r_2 + 2$ and $s = -1, 1$ are simple poles of the integrand function. 
Now,  considering the above two cases and applying Cauchy's residue theorem, we have
\begin{align}\label{Int_along_contour}
\frac{1}{2\pi i}\int_{\mathcal{C}} \Lambda_{\mathbb{F}, k}(s) y^{-s} {\rm d}s  = \mathcal{R}, 
\end{align}
where $\Lambda_{\mathbb{F}, k}(s) =  \Gamma(s)^d\zeta_{\mathbb{F}}(s)\zeta_{\mathbb{F}}(s+2k+1)  (2\pi)^{-s} D^s $ and the residual term is given by
\begin{align}\label{residues for different k}
\mathcal{R} = \begin{cases}
 R_0^{(r_2+1)}(y) + R_1^{(1)}(y) + R_{-2k}^{(r_2+1)}(y) + R_{-2k-1}^{(1)}(y)  + {\displaystyle \sum_{j=1}^k R_{-(2j-1)}^{(r_1+r_2)}(y) }\\  \hspace{8.2cm} + {\displaystyle \sum_{j=1}^{k-1} R_{-2j}^{(r_2)}(y)}, & \text{for} \, \, \, k > 0, \\
R_0^{(1)}(y) + R_1^{(1)}(y) + R_{-2k}^{(1)}(y), & \text{for} \, \, \, k < 0, \\
%R_0^{(r_2+2)} + R_1^{(1)} + R_{-1}^{(1)}, &  \text{for} \, \, \, k = 0,
\end{cases}
\end{align}
where $R_\gamma^{(l)}(y)$ denotes the residue of $ \Lambda_{\mathbb{F}, k}(s) y^{-s} $ of order $l$ at $s= \gamma$.  Letting $T \rightarrow \infty$ and making use of Stirling's formula for $\Gamma(s)$, we can show that both the horizontal integrals vanish.  Now,  after taking the left vertical integral to the right side and using \eqref{sum_to_integral} in \eqref{Int_along_contour},  we are left with
\begin{align}\label{Cauchy}
\sum_{n=1}^{\infty}\sigma_{\mathbb{F}, -2k -1}(n)V\left(\frac{(2\pi)^d y}{D}\bigg| \bar{0}_d \right) &= \mathcal{R} + \frac{1}{2\pi i}\int_{(\alpha)} \Lambda_{\mathbb{F}, k}(s) y^{-s} {\rm d}s. 
\end{align}
Now our main aim is to simplify the above vertical integral,  which we defined as
\begin{align}
I_{\mathbb{F}, k}(y) := \frac{1}{2\pi i}\int_{(\alpha)}\Lambda_{\mathbb{F}, k}(s)y^{-s}{\rm d}s.  
\end{align}
%where $\Lambda_{\mathbb{F}, k}(s)$ is defined as in \ref{Functional Equation for Product of Zeta Functions}.
%= \Gamma(s)^d \zeta_{\mathbb{F}}(s)\zeta_{\mathbb{F}}(s+2k+1)\left(\frac{(2\pi)^d}{D}\right)^{-s}.$
Replace $s$ by $-t-2k$,   so that $\Re(t) = -\Re(s)-2k = -\alpha - 2k > \max\{1, -2k\}.$ Let $\beta = \Re(t) = -\alpha - 2k.$ After this change of variable,  the above integral becomes
\begin{align*}
I_{\mathbb{F}, k}(y) &= \frac{1}{2\pi i}\int_{(\beta)}\Lambda_{\mathbb{F},k}(-t-2k)y^{t+2k}\rm {d}t.
\end{align*}
Now we bring in use of Lemma \ref{Functional Equation for Product of Zeta Functions} to write as 
\begin{align}
I_{\mathbb{F}, k}(y) &= (-1)^{kr_1+r_2} y^{2k} \frac{1}{2\pi i}\int_{(\beta)}\Lambda_{\mathbb{F},k}(t)y^{t}{\rm d}t \nonumber \\
&= (-1)^{(kr_1+r_2)}y^{2k}\sum_{n=1}^{\infty}\sigma_{\mathbb{F}, -2k -1}(n)V\left(\frac{(2\pi)^d n}{Dy} \bigg| \bar{0}_d \right), \label{left vertical}
\end{align}
where we have used \eqref{sum_to_integral} in the last step.  Now substituting \eqref{left vertical} in \eqref{Cauchy},  we obtain
\begin{align}\label{Main thm without residue}
\sum_{n=1}^{\infty}\sigma_{\mathbb{F}, -2k -1}(n)V\left(\frac{(2\pi)^d ny}{D} \bigg| \bar{0}_d \right) &=  (-1)^{k r_1+r_2}y^{2k}\sum_{n=1}^{\infty}\sigma_{\mathbb{F}, -2k -1}(n)V\left(\frac{(2\pi)^d n}{Dy} \bigg| \bar{0}_d \right) + \mathcal{R}.
\end{align}
Now we are left with calculating the term $\mathcal{R}$ containing all the residual terms.  
%Let $Y = \frac{(2\pi)^d y}{D}.$

{\bf Case 1}: Let $k > 0$.  In this case we need to calculate the following terms:
\begin{align*}
\mathcal{R} = R_0^{(r_2+1)}(y) + R_1^{(1)}(y) + R_{-2k}^{(r_2+1)}(y) + R_{-2k-1}^{(1)}(y) + \sum_{j=1}^kR_{-(2j-1)}^{(r_1+r_2)}(y) + \sum_{j=1}^{k-1} R_{-2j}^{(r_2)}(y).
\end{align*}
Using the definition of residue calculation,  one can check that 
\begin{align}
R_0^{(r_2+1)} (y) &= \frac{1}{(r_2)!}\lim_{s \rightarrow 0} \frac{{\rm d}^{r_2}}{{\rm d}s^{r_2}}\left(s^{r_2+1}    \Lambda_{\mathbb{F},k}(s) y^{-s}\right),  \label{residue at 0}
 \\
R_1^{(1)} (y) &= \lim_{s \rightarrow 1} (s-1)  \Lambda_{\mathbb{F},k}(s) y^{-s} \nonumber \\
& = \lim_{s \rightarrow 1} (s-1) \Gamma(s)^d\zeta_{\mathbb{F}}(s)\zeta_{\mathbb{F}}(s+2k+1)\left(\frac{(2\pi)^d y}{D}\right)^{-s}  \nonumber \\
& = H_{\mathbb{F} }  \zeta_{\mathbb{F}}(2k+2)  \frac{ D }{ (2\pi)^d y},   \label{residue at 1} 
\end{align}
in the last step we have used the class number formula \eqref{Class number formula}.  Further,  we have
\begin{align}\label{residue at -2k}
R_{-2k}^{(r_2+1)}(y) &= \frac{1}{(r_2)!}\lim_{s \rightarrow -2k} \frac{{\rm d}^{r_2}}{{\rm d}s^{r_2}}\left((s+2k)^{r_2+1}\Lambda_{\mathbb{F}, k}(s)y^{-s}\right). 
\end{align}
Now using Lemma \ref{Functional Equation for Product of Zeta Functions}, the above residual term can be rewritten as follows
{\allowdisplaybreaks \begin{align}
R_{-2k}^{(r_2+1)}(y) &= \frac{(-1)^{kr_1+r_2}}{(r_2)!}\lim_{s \rightarrow -2k} \frac{{\rm d}^{r_2}}{{\rm d}s^{r_2}}\left((s+2k)^{r_2+1}\Lambda_{\mathbb{F}, k}(-s-2k)y^{-s}\right), \nonumber \\
&= \frac{(-1)^{kr_1+r_2+1}y^{2k}}{(r_2)!}\lim_{s \rightarrow 0} \frac{{\rm d}^{r_2}}{{\rm d}s^{r_2}}\left(s^{r_2+1}\Lambda_{\mathbb{F}, k}(s)y^{s}\right),  \nonumber \\
&= (-1)^{kr_1+r_2+1}y^{2k}R_0^{(r_2+1)}\left(\frac{1}{y}\right). \label{connection at 0 and -2k}
\end{align}}
This shows that the residue at $s=0$ and $s=-2k$ are linked by the above relation.  Next,    the residue at $s=-2k-1$ is given by
\begin{align*}
R_{-2k-1}^{(1)}(y) &= \lim_{s\rightarrow -2k-1}(s+2k+1)\Lambda_{\mathbb{F}, k}(s)y^{-s}.
\end{align*}
Further,  use Lemma \ref{Functional Equation for Product of Zeta Functions} to see that
\begin{align}
R_{-2k-1}^{(1)}(y) &= (-1)^{kr_1+r_2}\lim_{s\rightarrow -2k-1}(s+2k+1)\Lambda_{\mathbb{F}, k}(-s-2k)y^{-s} \nonumber \\
&= (-1)^{kr_1+r_2+1}y^{2k}\lim_{s\rightarrow 1}(s-1)\Lambda_{\mathbb{F}, k}(s)y^{s} \nonumber\\
&= (-1)^{kr_1+r_2+1}y^{2k}R^{(1)}_1\left(\frac{1}{y}\right)  \label{connection at 1 and -(2k+1)} \\
&=  (-1)^{kr_1+r_2+1}y^{2k}  H_{\mathbb{F} }  \zeta_{\mathbb{F}}(2k+2)  \frac{ D y }{ (2\pi)^d }.  \label{residue at -(2k+1)}
\end{align}
The above relation indicates the connection between the residues at $s=1$ and $s=-2k-1$.  
The remaining residues at $s=-(2j-1)$ and $s=-2j$ can be evaluated as follows:
\begin{align}
R_{-(2j-1)}^{(r_1+r_2)} (y) &= \frac{1}{(r_1+r_2-1)!}\lim_{s \rightarrow -(2j-1)} \frac{{\rm d}^{r_1+r_2-1}}{{\rm d}s^{r_1+r_2-1}}\left((s+2j-1)^{r_1+r_2}\Lambda_{\mathbb{F}, k}(s)y^{-s}\right),  \nonumber \\
&= \frac{1}{(r)!}\lim_{s \rightarrow -(2j-1)} \frac{{\rm d}^{r}}{{\rm d}s^{r}}\left((s+2j-1)^{r+1}\Lambda_{\mathbb{F}, k}(s)y^{-s}\right),  \label{residue at -(2j-1)}\\
R_{-(2j)}^{(r_2)}(y) &= \frac{1}{(r_2-1)!}\lim_{s \rightarrow -2j} \frac{{\rm d}^{r_2-1}}{{\rm d}s^{r_2-1}}\left((s+2j)^{r_2} \Lambda_{\mathbb{F}, k}(s)y^{-s}\right). \label{residue at -2j}
\end{align}
As we know the poles of  $ \Lambda_{\mathbb{F}, k}(s)y^{-s}$ at $s=-(2j-1)$ and $s=-2j$ are of higher order,  so it is a difficult task to simplify the above residual terms.   However,  one can say that these residual terms will be polynomials in $\mathbb{C}[y,  \log(y)]$ with degree being $2j+r_1+r_2-2$ and $2j+r_2-1$,  respectively.  This is because for the residual term \eqref{residue at -(2j-1)} the highest degree of $y$ will be $2j-1$ and the highest power of $\log(y)$ will be $r=r_1+r_2-1$.  Similarly,  for the residual term \eqref{residue at -2j} the highest degree of $y$ will be $2 j$ and the highest degree of $\log(y)$ will be $r_2-1$.  
In a similar fashion,    one can verify that the residue at $s=0$,  i.e.,  the residual term \eqref{residue at 0} is a polynomial in $\mathbb{C}[\log(y)]$ of degree $r_2$.  
This completes the calculations of all residues for $k>0$.  Now we shall compute residues for $k<0$.  \\
{\bf Case 2}: When $k < 0$,  from \eqref{residues for different k},  we know that the residual term $\mathcal{R}$ is given by 
\begin{align*}
\mathcal{R}= R_0^{(1)}(y) + R_1^{(1)}(y) + R_{-2k}^{(1)}(y). 
\end{align*}
Recall that,  in this case,  poles of $\Lambda_{\mathbb{F}, k}(s)$ are at $0, 1, -2k$ and all are simple.  Therefore,   we have
 \begin{align}\label{R_0^1}
R_0^{(1)}(y) &= \lim_{s \rightarrow 0}~s \Lambda_{\mathbb{F}, k}(s)y^{-s} 
= \lim_{s \rightarrow 0} s\Gamma(s)^{d} \zeta_{\mathbb{F}}(s)\zeta_{\mathbb{F}}(s+2k+1) \left(\frac{(2\pi)^d y}{D}\right)^{-s}  \nonumber \\
& = \lim_{s \rightarrow 0}( s\Gamma(s))^{d}\frac{\zeta_{\mathbb{F}}(s)}{s^{r_1 + r_2 - 1}} \frac{\zeta_{\mathbb{F}}(s+2k+1)}{s^{r_2}} \left(\frac{(2\pi)^d y}{D}\right)^{-s} \nonumber \\
&= \lim_{s \rightarrow 0} \Gamma(s+1)^{d}\frac{\zeta_{\mathbb{F}}(s)}{s^{r_1 + r_2 - 1}} \frac{\zeta_{\mathbb{F}}(s+2k+1)}{s^{r_2}}  \left(\frac{(2\pi)^d y}{D}\right)^{-s} \nonumber \\
%&= \frac{\zeta_{\mathbb{F}}^{(r_1 + r_2 - 1)}(0)}{(r_1 + r_2 - 1)!}\times \frac{\zeta_{\mathbb{F}}^{(r_2)}(2k+1)}{(r_2)!} \nonumber \\
& =  \frac{C_\mathbb{F} \zeta_{\mathbb{F}}^{(r_2)}(2k+1)}{(r_2)!}.  
\end{align}
To obtain the final step we used \eqref{Laurent series_at s=0_1st coeff} 
%the fact that $s=0$ is a zero of order $r_1+r_2-1$ of $\zeta_{\mathbb{F}}(s)$ 
and $s=0$ is a zero of order $r_2$ for $\zeta_{\mathbb{F}}(s+2k+1)$ as $k<0$ and negative odd integers are zeros of $\zeta_{\mathbb{F}}(s)$ of order $r_2$.  
Now we calculate the residual term,  at $s=-2k$,  which is given by
\begin{align*}
R_{-2k}^{(1)}(y) &= \lim_{s \rightarrow -2k} (s+2k)\Lambda_{\mathbb{F}, k}(s)y^{-s}.  
\end{align*}
To obtain a simplified form,  we use Lemma \ref{Functional Equation for Product of Zeta Functions}.  Thus,  we get
\begin{align}\label{R_(-2k)^1}
R_{-2k}^{(1)}(y) &= (-1)^{kr_1+r_2}\lim_{s \rightarrow -2k}(s+2k)\Lambda_{\mathbb{F}, k}(-s-2k)y^{-s}  \nonumber \\
&= (-1)^{kr_1+r_2+1}y^{2k}\lim_{s \rightarrow 0} ~s \Lambda_{\mathbb{F}, k}(s)y^{s}  \nonumber \\
&=(-1)^{kr_1+r_2+1}y^{2k} R_0^{(1)}\left(\frac{1}{y}\right) \nonumber \\
%&= (-1)^{kr_1+r_2+1}y^{2k}  \frac{\zeta_{\mathbb{F}}^{(r_1 + r_2 - 1)}(0)}{(r_1 + r_2 - 1)!}\times \frac{\zeta_{\mathbb{F}}^{(r_2)}(2k+1)}{(r_2)!}, \nonumber \\
&= (-1)^{kr_1+r_2+1}y^{2k}   \frac{C_\mathbb{F} \zeta_{\mathbb{F}}^{(r_2)}(2k+1)}{(r_2)!}, 
\end{align}
here in the penultimate step and final step,  we have used the definition of $R_{0}^{1}(y)$ and its final expression \eqref{R_0^1}.  This indicates that the residues at $s=0$ and $s=-2k$ are associated with each other.  Thus,  from \eqref{R_0^1} and \eqref{R_(-2k)^1},  we have 
\begin{align}\label{sum of two residues}
R_0^{(1)}(y) + R_{-2k}^{(1)}(y)  = \left\{ 1+ (-1)^{kr_1+r_2+1}y^{2k} \right\}    \frac{C_\mathbb{F}  \zeta_{\mathbb{F}}^{(r_2)}(2k+1)}{(r_2)!}.  
\end{align}
Finally,  the residue at $s=1$,  a simple pole of $ \Lambda_{\mathbb{F}, k}(s)y^{-s}$,  can be evaluated as follows: 
{\allowdisplaybreaks \begin{align}\label{R_1^1}
R_1^{(1)} (y) = \lim_{s \rightarrow 1} (s-1)  \Lambda_{\mathbb{F},k}(s) y^{-s} &  = \lim_{s \rightarrow 1} (s-1) \Gamma(s)^d\zeta_{\mathbb{F}}(s)\zeta_{\mathbb{F}}(s+2k+1)\left(\frac{(2\pi)^d y}{D}\right)^{-s}  \nonumber \\
& = H_{\mathbb{F} }  \zeta_{\mathbb{F}}(2k+2)  \frac{ D }{ (2\pi)^d y},  \nonumber \\
&= \begin{cases}
 -\frac{1}{4\pi y},  & \quad \text{if} \, \, \, (k,r_1,  r_2) = (-1,1,  0), \\
  \frac{ H_{\mathbb{F} }  \zeta_{\mathbb{F}}(0) D }{ (2\pi)^2 y} ,  & \quad \text{if} \, \, \, (k,r_1,  r_2) = (-1,0,  1), \\
0,  & \quad \text{otherwise.}
\end{cases}
\end{align}}
Till now,  we assumed that $y>0$.  However,  by analytic continuation,  one can show that the identity can be extended for $\Re(y) > 0$.  Thus, we substitute $y = -iz$ with $z \in \mathbb{H}$ in \eqref{Main thm without residue} to get the following form,  for $z \in \mathbb{H}$,  
\begin{align}
\sum_{n=1}^{\infty}\sigma_{\mathbb{F}, -2k -1}(n)V\left(\frac{-(2\pi)^d n i z}{D} \bigg| \bar{0}_d \right) &=  (-1)^{k( r_1+1)+ r_2} z^{2k}\sum_{n=1}^{\infty}\sigma_{\mathbb{F}, -2k -1}(n)V\left(\frac{(2\pi)^d n i}{Dz} \bigg| \bar{0}_d \right) \nonumber \\
& + \mathcal{R},  \nonumber 
\end{align}
where,  for $k>0$,  the term $ \mathcal{R}$ is given by
\begin{align*}
\mathcal{R} & = \Big\{ R_0^{(r_2+1)}(-i z)  +  (-1)^{k(r_1+1)+r_2+1} z^{2k}R_0^{(r_2+1)}\left(\frac{i}{  z}\right) \Big\} \\
& + \Big\{  R_1^{(1)} (-i z) + (-1)^{k(r_1+1)+r_2+1} z^{2k}  R_1^{(1)} \left( \frac{i}{ z}\right)  \Big\} \nonumber  \\
& + {\displaystyle \sum_{j=1}^k R_{-(2j-1)}^{(r_1+r_2)}(-i z) + \sum_{j=1}^{k-1} R_{-2j}^{(r_2)}}(-i z).
\end{align*}
Here we note that to get the above form we have used the relations \eqref{connection at 0 and -2k} and \eqref{connection at 1 and -(2k+1)} among residues.   To simplicity further,   we define
\begin{align*}
& \mathfrak{R}_{0}(z):= R_0^{(r_2+1)} (- i z), 
\mathfrak{R}_{1}(z):= R_{1}^{(1)}(- i z),  \\
& \mathfrak{R}_{-(2j-1)}(z):= R_{-(2j-1)}^{(r_1+r_2)}(-i z),  \mathfrak{R}_{-(2j)}(z) := R_{-2j}^{(r_2)}(-i z).
\end{align*}
Then the above the term $\mathcal{R}$ becomes,  for $k >0$,
\begin{align*}
\mathcal{R}  = &  \Big\{ \mathfrak{R}_{0}(z) +  (-1)^{k(r_1+1)+r_2+1} z^{2k} \mathfrak{R}_{0}\left(-\frac{1}{z}\right)  \Big\} \\
 + & \Big\{ \mathfrak{R}_{1}(z) +  (-1)^{k(r_1+1)+r_2+1} z^{2k} \mathfrak{R}_{1}\left(-\frac{1}{z}\right)  \Big\} \\ 
 &  {\displaystyle \sum_{j=1}^k \mathfrak{R}_{-(2j-1)}(z) + \sum_{j=1}^{k-1} \mathfrak{R}_{-(2j)}(z). }
\end{align*}
In a similar way,  after substituting $y=-i z$,   for $k <0$,  from \eqref{R_0^1}-\eqref{sum of two residues},  the residual term $\mathcal{R}$ can be written as 
\begin{align*}
\mathcal{R}= & \Big\{ 1 +  (-1)^{k(r_1+1)+r_2+1} z^{2k}  \Big\}   \frac{C_\mathbb{F}\zeta_{\mathbb{F}}^{(r_2)}(2k+1)}{(r_2)!} \nonumber  \\
& + \mathfrak{R}(z), 
\end{align*}
where the term $ \mathfrak{R}(z)$ is given by 
\begin{align*}
 \mathfrak{R}(z)=  \begin{cases}
 -\frac{i}{4\pi z},  & \quad \text{if} \, \, \, (k,r_1,  r_2) = (-1,1,  0), \\
  \frac{ H_{\mathbb{F} }  \zeta_{\mathbb{F}}(0) D i }{ (2\pi)^2 z} ,  & \quad \text{if} \, \, \, (k,r_1,  r_2) = (-1,0,  1), \\
0,  & \quad \text{otherwise.}
\end{cases}
\end{align*}
Finally,  combining all these residual terms and using the definition \eqref{Infinite series F} of $\mathfrak{F}_{\mathbb{F}, k}(z)$,  we obtain the final identity.
\end{proof}

\begin{proof}[Corollary {\rm \ref{When F equals Q}}][]
If we consider $\mathbb{F}=\mathbb{Q}$,  then one can easily check that 
$$(r_1,  r_2,  d,  D,  H_{\mathbb{F}},  C_\mathbb{F})= (1,0,1,  1,  1,  -1/2).$$ Moreover,  the number field analogous generalized divisor function  $\sigma_{\mathbb{F}, -k}(n)$ reduces to the usual generalized divisor function $ \sigma_{-k}(n)$.  Therefore,  for $k >0$,  we have 
\begin{align*}
\mathfrak{F}_{\mathbb{F}, 2k+1}(z) & = \sum_{n=1}^\infty \sigma_{-(2k+1)}(n) e^{2 \pi i n z}, \\
 \mathfrak{S}_{\mathbb{F}, 2k+1}(z) & =  \sum_{n=1}^\infty \sigma_{-(2k+1)}(n) e^{2 \pi i n z}- \mathfrak{R}_0(z)- \mathfrak{R}_1(z),  
\end{align*}
where $\mathfrak{R}_0(z)= - \frac{\zeta(2k+1)}{2}$ and 
$\mathfrak{R}_1(z)= \frac{i \zeta(2k+2)}{2 \pi z} = \frac{(2 \pi i)^{2k+1}}{2z}  \frac{B_{2k+2}}{ (2k+2)!} $.  Again,  one can check that 
\begin{align*}
\mathfrak{R}_{-(2j-1)}(z) & = \lim_{s \rightarrow -(2j-1)} (s+2j-1)  \Gamma(s) \zeta(s) \zeta(s+2k+1)  (- 2\pi iz)^{-s}    \\
& = \frac{(2 \pi i)^{2k+1}}{2 z } \frac{B_{2j}}{(2j)!} \frac{B_{2k-2j+2}}{(2k-2j+2)! } z^{2j}.
\end{align*}
Here we have used the fact that $\zeta(1-2j)= - \frac{B_{2j}}{(2j)!}$ and Euler's formula \eqref{Euler's formula for even zeta values}.  Substituting these values in Theorem \ref{DB},  we get
\begin{align*}
& \sum_{n=1}^\infty \sigma_{-(2k+1)}(n) e^{2 \pi i n z} + \frac{\zeta(2k+1)}{2} -\frac{(2 \pi i)^{2k+1}}{2z}  \frac{B_{2k+2}}{ (2k+2)!} \\
& = z^{2k} \left\{  \sum_{n=1}^\infty \sigma_{-(2k+1)}(n) e^{-2 \pi i n/z} + \frac{\zeta(2k+1)}{2} +  \frac{(2 \pi i)^{2k+1}}{2z}  \frac{ z^2  B_{2k+2}}{   (2k+2)!}   \right\}  \\
& +  \frac{(2 \pi i)^{2k+1}}{2 z } \sum_{j=1}^{k} \frac{B_{2j}}{(2j)!} \frac{B_{2k-2j+2}}{(2k-2j+2)! } z^{2j}.
\end{align*}
This implies that
\begin{align*}
 \sum_{n=1}^\infty \sigma_{-(2k+1)}(n) e^{2 \pi i n z} + \frac{\zeta(2k+1)}{2}  & =  z^{2k} \left\{  \sum_{n=1}^\infty \sigma_{-(2k+1)}(n) e^{-2 \pi i n/z} + \frac{\zeta(2k+1)}{2} \right\}  \\
& +  \frac{(2 \pi i)^{2k+1}}{2 z } \sum_{j=0}^{k+1} \frac{B_{2j}}{(2j)!} \frac{B_{2k-2j+2}}{(2k-2j+2)! } z^{2j},
\end{align*}
which is exactly same as the identity \eqref{Grosswald identity}.  This completes the proof.  
\end{proof}

\begin{proof}[Theorem {\rm \ref{k>0 and real number field}}][]
Given that $\mathbb{F}$ is a totally real number field of degree $r_1$,  so we have $r_2=0$ and $d=r_1$.  Thus,  in Theorem \ref{DB},   we have 
\begin{align*}
\mathfrak{R}_{0}(z) &= \lim_{s \rightarrow 0} s\Gamma^{r_1}(s) \zeta_{\mathbb{F}}(s)\zeta_{\mathbb{F}}(s+2k+1)\left(-\frac{(2\pi)^{r_1}iz}{D}\right)^{-s}  \\
& =  \lim_{s \rightarrow 0} \Gamma^{r_1}(s+1) \frac{ \zeta_{\mathbb{F}}(s)}{ s^{r_1-1}}\zeta_{\mathbb{F}}(s+2k+1)\left(-\frac{(2\pi)^{r_1}iz}{D}\right)^{-s}  \\
& = C_{\mathbb{F}} \zeta_{\mathbb{F}}(2k+1),  \\
\mathfrak{R}_{1}(z) &= H_{\mathbb{F} }  \zeta_{\mathbb{F}}(2k+2)  \frac{ i D  }{ (2\pi)^{r_1} z},  \\
\mathfrak{R}_{-(2j-1)}(z) &= \frac{1}{(r_1-1)!}\lim_{s \rightarrow (1-2j)} \frac{{\rm d}^{r_1-1}}{{\rm d}s^{r_1-1}}\left((s+2j-1)^{r_1}\Lambda_{\mathbb{F}, k}(s)(-iz)^{-s}\right).
\end{align*}
Note that we have used \eqref{Laurent series_at s=0_1st coeff} to evaluate $\mathfrak{R}_{0}(z)$.  Moreover,  we point out that the terms $\mathfrak{R}_{-2j}(z)$ would not appear in this case since we are dealing with totally real fields.  Substituting the above values in Theorem \ref{DB},  we finish the proof of Theorem {\rm \ref{k>0 and real number field}}.   
\end{proof}

\begin{proof}[Corollary {\rm \ref{k>0 and quadratic real field}}][]
Given that $m$ is a square-free positive integer and $\mathbb{F}=\mathbb{Q}(\sqrt{m})$.   In this case,  we have $r_1=2,  r_2=0,  d=2$ and $D=4m$.  In this case,  we use \eqref{Bessel fn formula} to see that
\begin{align}\label{particular V}
V\left(-\frac{(2\pi)^2 niz}{4m} \bigg| \bar{0}_2 \right) = 2 K_0\left( 2 \pi  \sqrt{\frac{nz}{m}}e^{-\frac{i \pi}{4}}\right).
\end{align}
Moreover,  the residual terms become
\begin{align*}
\mathfrak{R}_{0}(z) 
& = \zeta_{\mathbb{F}}^{'}(0) \zeta_{\mathbb{F}}(2k+1),  \\
\mathfrak{R}_{1}(z) &= H_{\mathbb{F} }  \zeta_{\mathbb{F}}(2k+2)  \frac{ i m  }{  \pi^2 z},  \\
\mathfrak{R}_{-(2j-1)}(z) &=  \lim_{s \rightarrow (1-2j)} \frac{{\rm d}}{{\rm d}s}\left((s+2j-1)^2\Lambda_{\mathbb{F}, k}(s)(-iz)^{-s}\right).
\end{align*}
To calculate $\mathfrak{R}_{0}(z) $ we have used the fact $C_\mathbb{F}= \zeta_{\mathbb{F}}^{'}(0)$ as we are working on real quadratic field.  
%{\bf Can we simplify $ \mathfrak{R}_{-(2j-1)}(z)$ by writing $\zeta_{\mathbb{F}}(s)=\zeta(s) L(s,  \chi)$?}
Putting these terms in Theorem {\rm \ref{k>0 and real number field}},  one can complete the proof of this result.  
\end{proof}

\begin{proof}[Theorem {\rm \ref{k<0 and totally real fields}}][]
This is result is an implication of our main Theorem \ref{DB} for negative integers $k$ and totally real fields.  Mainly,  from \eqref{for k<0} with $r_2=0$,  we have
\begin{align*}
 \mathfrak{F}_{\mathbb{F}, 2k+1}(z) - C_\mathbb{F} \zeta_{\mathbb{F}}(2k+1) 
 & = (-1)^{k(r_1+1)}z^{2k} \left\{  \mathfrak{F}_{\mathbb{F}, 2k+1}\left(-\frac{1}{z}\right) - C_\mathbb{F}  \zeta_{\mathbb{F}}(2k+1) \right\} \nonumber \\
& + \begin{cases}
 -\frac{i}{4\pi z},  & \quad \text{if} \, \, \, (k,r_1,  r_2) = (-1,1,  0), \\
0,  & \quad \text{otherwise.}
\end{cases}
\end{align*}
Now replacing $k$ by $-k$ and simplifying,  we complete the proof.     
\end{proof}

\begin{proof}[Corollary {\rm \ref{Evaluation at z=i}}][]
Given that $k \geq 3$ and $r_1$ both are positive odd integers.  Substituting $z=i$ in Theorem {\rm \ref{k<0 and totally real fields}},  we obtain 
\begin{align*}
\mathfrak{F}_{\mathbb{F}, -2k+1}(i)  =C_\mathbb{F}  \zeta_{\mathbb{F}}(1-2k). 
\end{align*}
Now the proof of this corollary follows by using the value  \eqref{Laurent series_at s=0_1st coeff} of $C_\mathbb{F}$ and  the definition \eqref{Infinite series F} of $ \mathfrak{F}_{\mathbb{F}, -2k+1}(i) $. 
\end{proof}

\begin{proof}[Corollary {\rm \ref{Transformation formula for quadratic fields}}][]
% It is an application of Theorem {\rm \ref{k<0 and totally real fields}}.  
Substituting $\mathbb{F}=\mathbb{Q}(\sqrt{m})$,  i.e.,  $r_1=2,  r_2=0,  D=4m$ in  Theorem {\rm \ref{k<0 and totally real fields}},  we can see that 
\begin{align}\label{For k<0}
(i z)^{2k}\left\{  \mathfrak{F}_{\mathbb{F}, -2k+1}(z) - \zeta_{\mathbb{F}}^{'}(0)\zeta_{\mathbb{F}}(1-2k) \right\} 
& =  \left\{  \mathfrak{F}_{\mathbb{F}, -2k+1}\left(-\frac{1}{z}\right)
- \zeta_{\mathbb{F}}^{'}(0) \zeta_{\mathbb{F}}(1-2k) \right\}, 
\end{align}
where 
$$
\mathfrak{F}_{\mathbb{F}, -2k+1}(z) =  \sum_{n=1}^{\infty}  \sigma_{\mathbb{F}, 2k-1}(n)V\left(-\frac{(2\pi)^2 niz}{4m} \bigg| \bar{0}_2 \right).
$$
Now employing the expression \eqref{particular V} for $V\left(-\frac{(2\pi)^2 niz}{4m} \bigg| \bar{0}_2 \right)$ in the above series and then putting it in \eqref{For k<0},  the proof of Corollary {\rm \ref{Transformation formula for quadratic fields}} follows.  
% it follows that 
%\begin{align}
%(iz)^{2k} G_{\mathbb{F}, 2k}(z) = G_{\mathbb{F}, 2k}\left( -\frac{1}{z}\right), 
%\end{align}
%where 
%\begin{align*}
%G_{\mathbb{F}, 2k}(z) = 1 - \frac{2}{\zeta_{\mathbb{F}}'(0)\zeta_{\mathbb{F}}(1-2k)}\sum_{n=1}^{\infty}\sigma_{\mathbb{F}, 2k-1}(n)K_0\left(2\pi \sqrt{\frac{nz}{m}} e^{-\frac{i\pi}{4}}\right).
%\end{align*}
\end{proof}

\begin{proof}[Theorem {\rm \ref{k>0 and imaginary field}}][]
In this case,  we have taken $\mathbb{F}$ to be a purely imaginary number field with degree $d=2r_2$.  Thus,  the proof of this theorem immediately follows by substituting $r_1=0$ in Theorem \ref{DB}.  
\end{proof}

\begin{proof}[Corollary {\rm \ref{k>0 and quadratic imaginary field}}][]
Given that $\mathbb{F}=\mathbb{Q}(\sqrt{-m})$ is a quadratic imaginary field and $k>0$.    Therefore,  substituting $(r_1,  r_2,  d, D)= (0,1, 2,  4m)$ in Theorem  \ref{k>0 and imaginary field},  it yields that 
\begin{align*}
\mathfrak{S}_{\mathbb{F},  2k+1}(z) & = (-1)^{k+1}z^{2k} \mathfrak{S}_{\mathbb{F},  2k+1} \left(- \frac{1}{z}\right)   
 +  {\displaystyle \sum_{j=1}^k  \mathfrak{R}_{-(2j-1)}(z) + \sum_{j=1}^{k-1} \mathfrak{R}_{-2j}(z)} ,  
\end{align*}
where
\begin{align*}
\mathfrak{S}_{\mathbb{F},  2k+1}(z) &= \sum_{n = 1}^{\infty}\sigma_{\mathbb{F}, -2k-1}(n) V\left(-\frac{ \pi^2 niz}{m} \bigg| \bar{0}_2 \right) - \mathfrak{R}_{0}(z) - \mathfrak{R}_1(z),  
\end{align*}
and the residual terms are defined as 
{\allowdisplaybreaks \begin{align*}
\mathfrak{R}_{0}(z) &  =  \lim_{s \rightarrow 0} \frac{{\rm d}^{2}}{{\rm d}s^{2}}\left(s^{2}    \Lambda_{\mathbb{F},k}(s) (- i z) ^{-s}\right),     \nonumber \\
 \mathfrak{R}_1(z)  & = H_{\mathbb{F} }  \zeta_{\mathbb{F}}(2k+2)  \frac{ i m  }{ \pi^2 z},  \\
   \mathfrak{R}_{-(2j-1)}(z) & =   \lim_{s \rightarrow -(2j-1)}  \left((s+2j-1)\Lambda_{\mathbb{F}, k}(s) (- iz)^{-s}\right),  \nonumber \\ 
   & =  \lim_{s \rightarrow -(2j-1)}  \left((s+2j-1)^2 \Gamma(s)^2 \frac{\zeta_{\mathbb{F}}(s)}{s+2j-1} \zeta_{\mathbb{F}}(s+2k+1) \left( \frac{-i \pi^2 z}{m} \right)^{-s} \right) \nonumber \\
 &=   -\frac{\zeta_{\mathbb{F}}^{'}(1-2j)}{((2j-1)!)^2}\zeta_{\mathbb{F}}(2k-2j+2)\left(\frac{\pi^2 iz}{m}\right)^{2j-1}.  
%   \mathfrak{R}_{-2j}(z) & =  \frac{\zeta_{\mathbb{F}}^{'}(-2j)}{((2j)!)^2}\zeta_{\mathbb{F}}(2k-2j+1)\left(\frac{\pi^2 iz}{m}\right)^{2j}.
\end{align*}}
In a similar way,  one can show that 
\begin{align*}
\mathfrak{R}_{-2j}(z) & =  \frac{\zeta_{\mathbb{F}}^{'}(-2j)}{((2j)!)^2}\zeta_{\mathbb{F}}(2k-2j+1)\left(\frac{\pi^2 iz}{m}\right)^{2j}.
\end{align*}
Now combining all these residual terms and using the expression \eqref{particular V} for Steen function,  we complete the proof of \eqref{k>0 and quadratic imaginary field main eqn}.    
\end{proof}

\begin{proof}[Corollary {\rm \ref{For k=-1,  r_2=1}}][]
Substitute $k=-1$ and $\mathbb{F}=\mathbb{Q}(\sqrt{-m})$,  that is,  $r_1=0,  r_2=1$ in Theorem {\rm \ref{k>0 and imaginary field}}.  Thus,  from \eqref{for k<0_imaginary field}-\eqref{Residue at 1 for imag field},  we arrive at 
\begin{align}\label{k=-1_quad}
\mathfrak{U}_{\mathbb{F},  -1}(z) & = z^{-2} \mathfrak{U}_{\mathbb{F},  -1} \left(- \frac{1}{z}\right)   + \frac{H_{\mathbb{F}} \zeta_{\mathbb{F}}(0) i  m}{\pi^2 z},
\end{align}
where
\begin{align*}
\mathfrak{U}_{\mathbb{F},  -1}(z) =  \mathfrak{F}_{\mathbb{F}, -1}(z)  - \zeta_{\mathbb{F}}(0)  \zeta_{\mathbb{F}}^{'}(-1).  
\end{align*}
Here we used  \eqref{Laurent_0} with the fact that $C_\mathbb{F}=\zeta_\mathbb{F}(0)$ as $\mathbb{F}$ is quadratic imaginary field.  
Now multiplying by $z^2$ on both sides of \eqref{k=-1_quad},  we finish the proof of \eqref{U at k=-1}.   
Further,  to obtain \eqref{Exact evaluation at z=i},  substituting $z=i$ in \eqref{k=-1_quad} it yields that
\begin{align*}
 \mathfrak{F}_{\mathbb{F}, -1}(i)  =  \zeta_{\mathbb{F}}(0)  \zeta_{\mathbb{F}}^{'}(-1) +  \frac{H_{\mathbb{F}} \zeta_{\mathbb{F}}(0)   m}{ 2 \pi^2 }.
\end{align*}
Now using the definition \eqref{Infinite series F} of $ \mathfrak{F}_{\mathbb{F}, -1}(i)$ we derive that 
\begin{align*}
\sum_{n=1}^\infty \sigma_{\mathbb{F},1}(n) V\left( \frac{n \pi^2 }{m} | \bar{0}_2  \right) =\zeta_{\mathbb{F}}(0)  \zeta_{\mathbb{F}}^{'}(-1) + \frac{H_{\mathbb{F}} \zeta_{\mathbb{F}}(0)   m}{ 2 \pi^2 }.
\end{align*}
Finally,  using the fact that $V\left( \frac{n \pi^2 }{m} | \bar{0}_2  \right) = 2 K_0\left( 2\pi \sqrt{\frac{n}{m}}  \right)$,  we complete the proof of \eqref{Exact evaluation at z=i}.
\end{proof}

\begin{proof}[Corollary {\rm \ref{k<0 and imaginary field}}][]
The proof of this corollary immediately follows from  \eqref{for k<0_imaginary field}-\eqref{Residue at 1 for imag field} by replacing $k$ by $-k$.  As it is given that $(k,  r_2) \neq (1,1)$,  so  we have to use the fact that $\mathfrak{R}(z)=0$ and multiply by $z^{2k}$ on both sides of  \eqref{for k<0_imaginary field} after replacing $k$ by $-k$. 
\end{proof}

\begin{proof}[Theorem {\rm \ref{k=0 case}}][]
The proof of this theorem goes along the same line as in Theorem \ref{DB},  however,  we give brief outline.  This identity is due to the case $k=0$.  
In the way as we proceeded in \eqref{1st_Integral} and \eqref{sum_to_integral},  one can show that 
\begin{align}\label{sum_to_integral_k=0}
\sum_{n=1}^{\infty}\sigma_{\mathbb{F}, -1}(n)V\left(\frac{(2\pi)^d ny}{D} \bigg| \bar{0}_d \right) &= \frac{1}{2\pi i}\int_{(c)} \Gamma(s)^d \zeta_{\mathbb{F}}(s)\zeta_{\mathbb{F}}(s +1)\left(\frac{(2\pi)^d y}{D}\right)^{-s} {\rm d}s \nonumber \\
& = \frac{1}{2\pi i}\int_{(c)}   \Lambda_{\mathbb{F}, 0}(s) y^{-s} {\rm d}s,  
\end{align}
where $\Lambda_{\mathbb{F}, 0}(s)=  \Gamma(s)^d \zeta_{\mathbb{F}}(s)\zeta_{\mathbb{F}}(s +1) (2\pi)^{- d s} D^{s} $,  which is exactly same as we defined in  \eqref{Lamda_F_k},  and $ 1<  \Re(s)= c < 1 + \epsilon$ with $0< \epsilon<1$.  
%{\bf Write about contour and add diagram.  Discuss order of the poles of the integrand function.}
Further proceeding in a similar way as in Theorem \ref{DB},  here also we set up a rectangular contour $\mathcal{C}$ made up of the vertices $c-iT, c+iT, \alpha+iT$ and $\alpha-iT$,  with $-2 <\alpha < -1$ and $T$ is some large positive quantity.
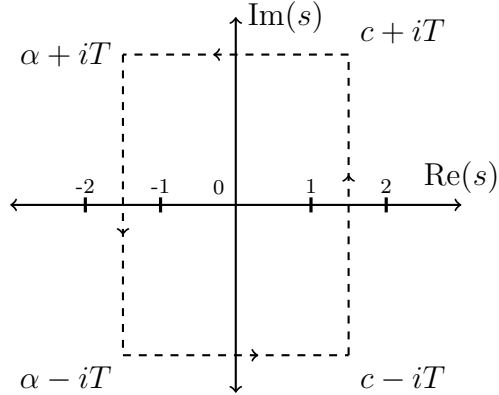
{\begin{center}
\begin{tikzpicture}[very thick,decoration={
  markings,
  mark=at position 0.6 with {\arrow{>}}}
 ] 
 \draw[thick,dashed,postaction={decorate}] (-1.5,-2)--(1.5,-2) node[below right, black] {$c-i T$};
 \draw[thick,dashed,postaction={decorate}] (1.5,-2)--(1.5,2) node[above right, black] {$c+iT$} ;
 \draw[thick,dashed,postaction={decorate}] (1.5,2)--(-1.5,2) node[left, black] {$\alpha+i T$}; 
 \draw[thick,dashed,postaction={decorate}] (-1.5,2)--(-1.5,-2) node[below left, black] {$\alpha-i T$}; 
 \draw[thick, <->] (-3,0) -- (3,0) coordinate (xaxis);
 \draw[thick, <->] (0,-2.5) -- (0,2.5)node[midway, above left, black] {\tiny0} coordinate (yaxis);
 \draw (1,0.1)--(1,-0.1) node[midway, above, black]{\tiny 1};
 \draw (-1,0.1)--(-1,-0.1) node[midway, above, black]{\tiny -1};
 \draw (2,0.1)--(2,-0.1) node[midway, above, black]{\tiny 2};
  \draw (-2,0.1)--(-2,-0.1) node[midway, above, black]{\tiny -2};
 \node[above] at (xaxis) {$\mathrm{Re}(s)$};
 \node[right] at (yaxis) {$\mathrm{Im}(s)$};

\end{tikzpicture}
\end{center}}
We now examine the poles of our integrand function $\Lambda_{\mathbb{F}, 0}(s) $.  
%Recall that $\zeta_\mathbb{F}(s)$ has a zero of order $r_1 + r_2 - 1$ at $s=0$, zero of order $r_2$ at $s=-1$ and a simple pole at $s=1$.  
At $s=0$,  $\Gamma(s)^d$ has a pole order $d$,  $\zeta_{\mathbb{F}}(s)$ has a zero of order $r_1+r_2-1$,  and $\zeta_{\mathbb{F}}(s+1)$ has a simple pole.  Therefore,  the order of the pole of the integrand function $ \Lambda_{\mathbb{F}, 0}(s)$ at $s=0$ is $d-(r_1+r_2-1)+1= r_2+2$ since $d=r_1+2r_2$.  It is easy to see that the integrand function has a simple pole at $s=1$.  From Lemma \ref{Functional Equation for Product of Zeta Functions},  we can see that 
$\Lambda_{\mathbb{F}, 0}(s)= (-1)^{r_2}\Lambda_{\mathbb{F}, 0}(-s)$.  This indicates that $s=-1$ is also a simple pole of the integrand function.  
%Hence the integrand function has a simple pole at $s=1$  (due to $\zeta_\mathbb{F}(s)$).  At $s=-1$,  $\Gamma(s)^d$
%
%
%Hence the integrand function has a simple pole at $s=1$ (due to $\zeta_\mathbb{F}(s)$) and $s=-1$ (due to pole of order $d$ of $\Gamma(s)^d$ and zero of order $r_1 + 2r_2 - 1$ of $\zeta_\mathbb{F}(s)\zeta_\mathbb{F}(s+1)$), and a pole of order $r_2 + 2$ at $s=0$ (due to pole of order $d$ of $\Gamma(s)^d$, simple pole of $\zeta_\mathbb{F}(s+1)$ and zero of order $r_1 + r_2 - 1$ of $\zeta_\mathbb{F}(s)$).
Now utilizing Cauchy residue theorem,  we arrive 
\begin{align}\label{after CRT}
\sum_{n=1}^{\infty}\sigma_{\mathbb{F}, -1}(n)V\left(\frac{(2\pi)^d ny}{D} \bigg| \bar{0}_d \right) =  (-1)^{r_2} \sum_{n=1}^{\infty}\sigma_{\mathbb{F}, -1}(n)V\left(\frac{(2\pi)^d n}{D y} \bigg| \bar{0}_d \right) + \mathcal{R},
\end{align}
where the residual term $\mathcal{R}$ is given by 
\begin{align}\label{Residues for k=0}
\mathcal{R}=R_{0}^{(r_2 +2)}(y) + R_{1}^{(1)}(y) +  R_{-1}^{(1)}(y).
\end{align}
The above term $ \mathcal{R}$  makes the main difference with Theorem \ref{DB},  see \eqref{residues for different k}.  The terms of $ \mathcal{R}$ can be calculated as follows:
 \begin{align}
R_{0}^{(r_2 +2)}(y) &= \frac{1}{(r_2+1)!} \lim_{s \rightarrow 0} \frac{{\rm d}^{r_2+1}}{{\rm d}s^{r_2+1}} \left(s^{r_2+2} \Lambda_{\mathbb{F}, 0}(s) y^{-s}  \right),  \\
R_{1}^{(1)}(y) &=  \lim_{s \rightarrow 1}  (s-1) \Lambda_{\mathbb{F}, 0}(s) y^{-s}   \nonumber \\
& =  \lim_{s \rightarrow 1}  (s-1) \Gamma(s)^d \zeta_{\mathbb{F}}(s)\zeta_{\mathbb{F}}(s +1) (2\pi)^{-d s} D^{s} y^{-s}  \nonumber  \\
& = \frac{H_{\mathbb{F}} \zeta_{\mathbb{F}}(2) D}{(2\pi)^d y}, \\
R_{-1}^{(1)}(y) &=   \lim_{s \rightarrow -1}   (s+1) \Lambda_{\mathbb{F}, 0}(s) y^{-s}   \nonumber \\
&=   \lim_{s \rightarrow -1}  (s+1) (-1)^{r_2} \Lambda_{\mathbb{F}, 0}(-s) y^{-s}  \nonumber \\
 &=  (-1)^{r_2+1}  \lim_{s \rightarrow 1}  (s-1)  \Lambda_{\mathbb{F}, 0}(s) y^{s}  \nonumber \\
 & = (-1)^{r_2+1}  \frac{H_{\mathbb{F}} \zeta_{\mathbb{F}}(2) D y}{(2\pi)^d }.  
\end{align}
Here  we used class number formula \eqref{Class number formula} and Lemma \ref{Functional Equation for Product of Zeta Functions} to simplify the above residual terms $R_{1}^{(1)}(y)$ and $R_{-1}^{(1)}(y)$.  Substituting the above residual terms in \eqref{Residues for k=0} 
and together with \eqref{after CRT},  we get
 \begin{align*}
& \sum_{n=1}^{\infty}\sigma_{\mathbb{F}, -1}(n)V\left(\frac{(2\pi)^d ny}{D} \bigg| \bar{0}_d \right) - \frac{H_{\mathbb{F}} \zeta_{\mathbb{F}}(2) D}{(2\pi)^d y} \nonumber \\
& =  (-1)^{r_2} \left\{ \sum_{n=1}^{\infty}\sigma_{\mathbb{F}, -1}(n)V\left(\frac{(2\pi)^d n}{D y} \bigg| \bar{0}_d \right) -   \frac{H_{\mathbb{F}} \zeta_{\mathbb{F}}(2) D y}{(2\pi)^d }  \right\} \nonumber \\
& +  \frac{1}{(r_2+1)!} \lim_{s \rightarrow 0} \frac{{\rm d}^{r_2+1}}{{\rm d}s^{r_2+1}}  \left(  s^{r_2+2} \Lambda_{\mathbb{F}, 0}(s) y^{-s}\right). 
\end{align*}
The above identity holds for $y>0$,  however,
by analytic continuation,  one can show that this identity is also true for $\Re(y) > 0$.  We substitute $y = -iz$ to get the following form,  for $z \in \mathbb{H}$,  
{\allowdisplaybreaks \begin{align*}
& \sum_{n=1}^{\infty}\sigma_{\mathbb{F}, -1}(n)V\left(-\frac{(2\pi)^d n i z}{D} \bigg| \bar{0}_d \right) - \frac{H_{\mathbb{F}} \zeta_{\mathbb{F}}(2) i D}{(2\pi)^d  z} \nonumber \\
& =  (-1)^{r_2} \left\{ \sum_{n=1}^{\infty}\sigma_{\mathbb{F}, -1}(n)V\left(\frac{(2\pi)^d ni}{D z} \bigg| \bar{0}_d \right)  +  \frac{H_{\mathbb{F}} \zeta_{\mathbb{F}}(2) D i z}{(2\pi)^d }  \right\} \nonumber \\
& +  \frac{1}{(r_2+1)!} \lim_{s \rightarrow 0} \frac{{\rm d}^{r_2+1}}{{\rm d}s^{r_2+1}}  \left(  s^{r_2+2} \Lambda_{\mathbb{F}, 0}(s) (-i z)^{-s}\right). 
\end{align*}}
Now if we use the definition \eqref{Infinite series F} of $\mathfrak{F}_{\mathbb{F}, 1}(z)$,  then the above expression can be rewritten as 
\begin{align*}
 \mathfrak{F}_{\mathbb{F}, 1}(z) - \mathfrak{R}_1(z) = (-1)^{r_2} \left\{ \mathfrak{F}_{\mathbb{F}, 1} \left(- \frac{1}{z} \right) - \mathfrak{R}_1\left(-\frac{1}{z}  \right) \right\} +\mathfrak{R}_{0}(z),
\end{align*}
where 
\begin{align*}
\mathfrak{R}_1(z)=  \frac{H_{\mathbb{F}} \zeta_{\mathbb{F}}(2) i D}{(2\pi)^d  z},  \quad {\rm and} \quad \mathfrak{R}_{0}(z)=  \frac{1}{(r_2+1)!} \lim_{s \rightarrow 0} \frac{{\rm d}^{r_2+1}}{{\rm d}s^{r_2+1}}  \left(  s^{r_2+2} \Lambda_{\mathbb{F}, 0}(s) (-i z)^{-s}\right).
\end{align*}
Letting $\mathfrak{T}_{\mathbb{F}}(z):=  \mathfrak{F}_{\mathbb{F}, 1}(z) - \mathfrak{R}_1(z)$,  one can complete the proof of Theorem \ref{k=0 case}.
\end{proof}

\begin{proof}[Corollary {\rm \ref{k=0 totally real field}}][]
Letting $\mathbb{F}$ to be a totally real number field of degree $d=r_1$ in Theorem \ref{k=0 case},  we get 
\begin{align}\label{totally real}
\mathfrak{F}_{\mathbb{F}, 1}(z) - \mathfrak{F}_{\mathbb{F}, 1}\left(-\frac{1}{z}\right)& =  \mathfrak{R}_1(z) - \mathfrak{R}_1 \left(- \frac{1}{z} \right) + \mathfrak{R}_0(z),  
\end{align}
where $$\mathfrak{R}_1(z)=  \frac{H_{\mathbb{F}} \zeta_{\mathbb{F}}(2) i D}{(2\pi)^{r_1}  z}, \quad  \mathfrak{R}_0(z) = \lim_{s \rightarrow 0} \frac{{\rm d}}{{\rm d}s}  \left(  s^{2} \Gamma^{r_1} (s) \zeta_{\mathbb{F}}(s) \zeta_{\mathbb{F}}(s+1) \left(- \frac{(2 \pi)^{r_1} i z}{D}  \right)^{-s}\right).  $$
Now we shall try to simplify the term $\mathfrak{R}_0(z)$.  Utilizing the fact that
\begin{align*}
\Gamma(s)= \frac{1}{s}-\gamma + O(s),
\end{align*}
one can see that
\begin{align}\label{Laurent_gamma}
(s \Gamma(s))^{r_1}= 1 -  r_1 \gamma s + O(s^2).  
\end{align}
Since $\mathbb{F}$ is a totally real number field of degree $r_1$,  so from \eqref{Laurent_0}, the Laurent series expansion of $\zeta_{\mathbb{F}}(s) $ at $s=0$,  one has
\begin{align}\label{Laurent_at s=0}
\frac{\zeta_{\mathbb{F}}(s)}{s^{r_1-1}}= a_0 + a_1 s + O(s^{2}),
\end{align}
where 
$$
a_0 = \frac{\zeta_{\mathbb{F}}^{(r_1-1)}(0)  }{(r_1-1)!} =C_\mathbb{F}, \quad a_1= \frac{\zeta_{\mathbb{F}}^{(r_1)}(0)  }{(r_1)!}. 
$$
Further, utilizing \eqref{Laurent_1},  the Laurent series expansion of $ \zeta_{\mathbb{F}}(s)$, it yields that
\begin{align}\label{Laurent_Zeta_1}
s \zeta_{\mathbb{F}}(s+1)= H_{\mathbb{F}} + \gamma_{\mathbb{F}} s + O(s^2).  
\end{align}
Moreover,  we have
\begin{align}\label{Laurent_exp}
\left(- \frac{(2 \pi)^{r_1} i z}{D}  \right)^{-s} = 1 - s \log\left(- \frac{(2 \pi)^{r_1} i z}{D}  \right) + O(s^2).  
\end{align}
Now combining all the above Laurent series expansions \eqref{Laurent_gamma}-\eqref{Laurent_exp},  one can check that the coefficient of $s$ in the Laurent series expansion of $s^{2} \Gamma^{r_1} (s) \zeta_{\mathbb{F}}(s) \zeta_{\mathbb{F}}(s+1) (-2 \pi i z)^{-s}$ is $ a_0 \gamma_{\mathbb{F}} - a_0 H_{\mathbb{F}} \left\{ r_1 \gamma + \log\left(- \frac{(2 \pi)^{r_1} i z}{D}  \right) \right\}  + a_1 H_{\mathbb{F}}$,  which shows that
\begin{align*}
\mathfrak{R}_0(z)=  a_0 \gamma_{\mathbb{F}} - a_0 H_{\mathbb{F}} \left\{ r_1 \gamma + \log\left(- \frac{(2 \pi)^{r_1} i z}{D}  \right) \right\}  + a_1 H_{\mathbb{F}}.
\end{align*}
Finally,  substituting the above value of $\mathfrak{R}_0(z)$ in \eqref{totally real},  we complete the proof of \eqref{totally real field_k=0}. 

\end{proof}

\begin{proof}[Corollary {\rm \ref{eta identity}}][]
We know that $\zeta_{\mathbb{Q}}(s) = \zeta(s)$.   
Therefore,  considering $\mathbb{F} = \mathbb{Q}$,  
Corollary \ref{k=0 totally real field} reduces to 
\begin{align}\label{For F=Q}
 \mathfrak{F}_{\mathbb{Q}, 1}(z) - \mathfrak{F}_{\mathbb{Q}, 1}\left(-\frac{1}{z}\right) = \mathfrak{R}_1(z) - \mathfrak{R}_1 \left( -\frac{1}{z}\right) + \mathfrak{R}_0(z),
 \end{align}
 where 
 \begin{align*}
  \mathfrak{F}_{\mathbb{Q}, 1}(z) & =   \sum_{n=1}^{\infty}  \sigma_{\mathbb{Q}, -1}(n)V\left(-2\pi niz \bigg| \bar{0} \right) = \sum_{n=1}^{\infty}  \sigma_{-1}(n) \exp(2\pi i n z),  \\
  \mathfrak{R}_1(z)  & = \frac{H_{\mathbb{Q}} \zeta_{\mathbb{Q}}(2) i D}{ 2\pi  z} = \frac{i \pi}{12 z},    \\
  \mathfrak{R}_0(z) &=  a_0 \gamma_{\mathbb{Q}} - a_0 H_{\mathbb{Q}} \left\{ r_1 \gamma + \log(-2 \pi i z) \right\}  + a_1 H_{\mathbb{Q}}= \frac{1}{2}\log(-iz).  
 \end{align*}
The above simplified forms have been obtained by using the following well-known values:
$$
H_{\mathbb{Q}}=1,  \gamma_{\mathbb{Q}}=\gamma,  \zeta_{\mathbb{Q}}(2)=\frac{\pi^2}{6},  \\
 a_0 = \zeta(0)=-\frac{1}{2},  a_1= \zeta'(0)=-\frac{1}{2}\log(2\pi).   
$$
Eventually, putting the above values of $ \mathfrak{F}_{\mathbb{Q}, 1}(z),   \mathfrak{R}_1(z),    \mathfrak{R}_0(z) $ in \eqref{For F=Q} and using the fact that $n \sigma_{-1}(n)=\sigma(n)$,  one can finish the proof of \eqref{Equivalent eta identity}. 
% Moreover,  one can easily check that 
%$\mathfrak{R}_1(z) = \frac{i\pi}{12z}$.  From the definition,  we know
%\begin{align}
%\mathfrak{R}_{0}(z) & =   \lim_{s \rightarrow 0} \frac{{\rm d}}{{\rm d}s}  \left(  s^{2} \Gamma(s) \zeta(s) \zeta(s+1) (-2\pi i z)^{-s}\right).
%\end{align}
%Our claim is that 
%\begin{align}
%\mathfrak{R}_{0}(z) & =  \frac{1}{2}\log(-iz).  
%\end{align}
%Using the Laurent series expansion,  we know that 
%\begin{align}
%s \Gamma(s) & = 1 -  \gamma s + O(s^2),  \\
%\zeta(s)& = - \frac{1}{2} - \frac{1}{2}\log(2\pi) s + O(s^2), \\
%s \zeta(s+1) &= 1 + \gamma s+ O (s^2),  \\
%(-2 \pi i z)^{-s}  & = 1 - \log(-2 \pi i z) s + O (s^2).  
%\end{align}
%From the above these expansions,  one can easily check that the coefficient of $s$ in the Laurent series expansion of $s^{2} \Gamma(s) \zeta(s) \zeta(s+1) (-2\pi i z)^{-s}$ around $s=0$ is $\frac{1}{2}\log(-i z)$,  which proves our claim.  Now substituting the values of $\mathfrak{R}_{0}(z)$ and $\mathfrak{R}_1(z)$ in Theorem \ref{k=0 case},  the proof of \eqref{Equivalent eta identity} follows.  
\end{proof}

\section{Concluding Remarks}
Ramanujan's formula \eqref{Ramanujan's Formula} involves the following Lambert series
\begin{align}\label{Lambert series 1}
\sum_{n=1}^{\infty}\frac{n^{-k}}{e^{ny}-1} = \sum_{n=1}^{\infty}\sigma_{-k}(n)e^{-ny}, \quad \Re(y) > 0,
\end{align}
associated with the generalized divisor function $\sigma_{k}(n) = \sum_{d|n}d^{k}, k \in \mathbb{C}$.
Upon suitable change of variable, one can easily observe that the above series is exactly same as the following series: 
\begin{align}\label{Lambert series 2}
F_k(z) = \sum_{n=1}^{\infty}\sigma_{-k}(n)e^{2\pi inz}, \quad z \in \mathbb{H},
\end{align}
which is also present in Grosswald's identity \eqref{Grosswald identity}. 
Recently,  Banerjee, Gupta and Kumar \cite[Equation~(1.7)]{BGK23} have generalized the above Lambert series \eqref{Lambert series 1}, while obtaining a number field analogue of Ramanujan's identity,   in the following way:
$$
\sum_{\mathfrak{a} \subset \mathcal{O}_{\mathbb{K}}} \mathfrak{N}(\mathfrak{a})^k\Omega_{\mathbb{F}} \left(\frac{\mathfrak{N}(\mathfrak{a})y}{D}\right) =  \sum_{n=1}^{\infty}\mathtt{a}_{\mathbb{F}}(n)n^k\Omega_{\mathbb{F}}\left(\frac{ny}{D}\right),
$$
where the function $\Omega_{\mathbb{F}}(x) $,  which involves the Meijer $G$-function, is defined as
$$
\Omega_{\mathbb{F}}(x) := \frac{2^{1-r_1-r_2}}{\pi^{1-\frac{r_1}{2}}}\sum_{j=1}^{\infty}\mathtt{a}_{\mathbb{F}}(j)G_{0, 2d}^{d+1,0}\left(
\begin{array}{c}
- \\
(0)_{r_1+r_2}, \left(\frac{1}{2}\right)_{r_2+1};\left(\frac{1}{2}\right)_{r_1+r_2-1},(0)_{r_2}
\end{array}\bigg| \frac{x^2 j^2}{4^d}
\right),
$$
with $d = r_1 + 2r_2$ (the degree of the field $\mathbb{F}$ over $\mathbb{Q}$), and $D$ is the absolute value of the discriminant of $\mathbb{F}$.
% Letting $\mathbb{F} = \mathbb{Q}$, $\Omega_{\mathbb{F}}(x)$ simplifies to $\frac{1}{e^x - 1}$. 
In this paper,  we obtained a new number field analogue of Ramanujan's identity by 
generalizing the Lambert series $F_k(z)$, defined in \eqref{Lambert series 2},  in a different way than Banerjee et.  al.   Mainly, we studied modular transformation formula for the following infinite series: 
%We have used the following generalized version
\begin{align*}
\mathfrak{F}_{\mathbb{F}, k}(z) = \sum_{n=1}^{\infty}\sigma_{\mathbb{F}, -k}(n)V\left(-\frac{(2\pi)^d niz}{D} \bigg| \bar{0}_d \right),
\end{align*}
where $\sigma_{\mathbb{F},k}(n)$ is the number field analogue \eqref{Gupta_Pandit} of the generalized divisor function $\sigma_k(n)$ and $V(z| \bar{0}_d)$ is the Steen function \eqref{Voronoi Steen function}.  
%When $\mathbb{F} = \mathbb{Q}$,  it is easy to see that $\mathfrak{F}_{\mathbb{Q}, k}(z) = F_k(z)$.  
By obtaining a number field analogue of Ramanujan-Grosswald identity for odd zeta values, 
we are able to find a formula for Dedekind zeta function at odd arguments.
As an application our main identity i.e.,  Theorem \ref{DB},  we obtain many interesting modular transformation formulae that are true for totally real number fields and purely imaginary fields,  see Theorem \ref{k>0 and real number field},  Theorem \ref{k>0 and imaginary field}.  In particular,  for real quadratic fields and imaginary quadratic fields, we obtained modular transformation formulae,  i.e.,  Corollary \ref{Transformation formula for quadratic fields}, Corollary \ref{k<0 and imaginary field} that are perfect analogue of the transformation formula \eqref{Transformation2} for the Eisenstien series.  Further,  we found an exact evaluation of $\mathfrak{F}_{\mathbb{F}, 2k-1}(i)$ in terms of class number and the values of the Dedekind zeta function at negative odd integers,  see \eqref{Exact evaluation_totally real_at i},  \eqref{Exact evaluation for purely imag field} and \eqref{Exact evaluation_quad_imag_i}.   We also obtain a modular transformation formula which generalizes transformation formula for the Dedekind eta function $\eta(z)$,  see Theorem \ref{k=0 case}.  As one of the interesting applications of Theorem \ref{k=0 case},   we derived a formula \eqref{another formula for class number} for the class number of a totally real field which also indicates a relation with the Kronceker's limit formula for the Dedekind zeta function.  This observation might be of an independent interest to the readers of this article.

{\bf Acknowledgement} The authors would like to thank Prof.  Bruce Berndt for going through the manuscript and giving valuable suggestions.  The first author's research is supported by the Prime Minister Research Fellowship (PMRF),  Govt.  of India,  Grant No.  2102227. The last author wants to thank Science and Engineering Research Board (SERB),  India,   for giving MATRICS grant (File No. MTR/2022/000545) and SERB CRG grant (File No. CRG/CRG/2023/002122).  Both authors sincerely thank IIT Indore for providing conductive research environment.

\end{document}